\documentclass[english]{amsart}

\usepackage{esint}
\usepackage[svgnames]{xcolor} 
\usepackage[colorlinks,citecolor=red,pagebackref,hypertexnames=false,breaklinks]{hyperref}
\usepackage{pgf,tikz}
\usepackage{pdfsync}
\usepackage{mathabx}

\usepackage{dsfont}
\usepackage{url}
\usepackage[utf8]{inputenc}
\usepackage[T1]{fontenc}
\usepackage{lmodern}
\usepackage{babel}
\usepackage{mathtools}  
\usepackage{amssymb}
\usepackage{lipsum}
\usepackage{mathrsfs}
\usepackage{color}

\newtheorem{theorem}{Theorem}[section]
\newtheorem{proposition}{Proposition}[section]
\newtheorem{lemma}{Lemma}[section]

\newtheorem{corollary}{Corollary}[section]

\numberwithin{equation}{section}

\title[Inverse problems for evolution equations]{A unified approach to solving some inverse problems for evolution equations by using observability inequalities}

\thanks{The research of  MC and FT was supported in part by  grant
LabEx PERSYVAL-Lab (ANR-11-LABX- 0025-01) and grant 
ANR-17-CE40-0029 of the French National Research Agency ANR (project MultiOnde).}

\author[Ka\"{\i}s Ammari]{Ka\"{\i}s Ammari}
\address{Ka\"{\i}s Ammari, UR Analysis and Control of PDEs, UR 13ES64, Department of Mathematics, Faculty of Sciences of Monastir, University of Monastir, 5019 Monastir, Tunisia }
\email{kais.ammari@fsm.rnu.tn}

\author[Mourad Choulli]{Mourad Choulli}
\address{Universit\'e de Lorraine, 34 cours L\'eopold, 54052 Nancy cedex, France}
\email{mourad.choulli@univ-lorraine.fr}

\author[Faouzi Triki]{Faouzi Triki}
\address{Laboratoire Jean Kuntzmann,  UMR CNRS 5224, 
Universit\'e  Grenoble-Alpes, 700 Avenue Centrale,
38401 Saint-Martin-d'H\`eres, France}
\email{faouzi.triki@univ-grenoble-alpes.fr}

\date{}

\begin{document}

\begin{abstract}
We survey some of our recent results on inverse problems for evolution equations. The goal is to provide a unified approach to solve various types of evolution equations. The  inverse problems we consider consist in determining unknown coefficients from boundary measurements by varying  initial conditions. Based on observability inequalities and a special choice of  initial conditions, we provide uniqueness and stability estimates for the recovery of  volume and  boundary  lower order coefficients in wave and heat  equations. Some of the results presented here  are slightly improved from their original versions.

\end{abstract}

\subjclass[2010] {35R30}

\keywords{Evolution equations, Laplace-Beltrami operator, observability inequality, geometric control, initial-to-boundary operator}

\maketitle

\tableofcontents

\section{Introduction}
Inverse coefficient problems for evolution equations have been a very active  area in  mathematical and numerical research over the last decades, driven by numerous  applications.   They are intrinsically difficult to solve: this fact is due
 in part to their very mathematical structure and to the fact that generally only partial data is available \cite{Is}. 
We survey in this paper some of our recent results on inverse problems for evolution equations concerning  heat and 
wave equations. In \cite{ASTT} the authors proposed a general method to deal with inverse source problems for evolution equations. Starting from the ideas in \cite{ASTT}, we developed an approach based on observability inequalities and a spectral decomposition to solve some inverse coefficients problems in evolution equations \cite{AC1, AC3, ACT}.
However the approach is older than that.  Inverse coefficient problems in heat and wave equations using control techniques have been studied by a large community  of people (see for instance \cite{Ya, LT, IY, IY2,  RS, PY, BK, BaK} 
and the references therein). It would be impossible to present here all the relevant results that have been proved in this research direction.  We will  be mainly focusing  on the results  that are  closely connected to the  considered  inverse coefficient problems in heat and wave equations. 

The measurements  are made on a sub-boundary by varying initial conditions. The key idea in our analysis consists in reducing the inverse coefficients problems to inverse source problems. This is achieved by using a spectral decomposition
and unique continuation property of eigenfunctions. 

For simplicity convenience we limited ourselves to initial boundary value problems for wave and heat equations. But our analysis can be extended to other types of evolution equations such as dynamical Schr\"odinger equation.

The main ingredient in our approach is observability inequalities. We point out that the wave and the heat equations have different observability properties. We know that, under some appropriate conditions, the wave equation is exactly observable, while the heat equation is only final time observable \cite{TW, Zu}. We refer to Section 2 for details. In Section 3, we establish weighted interpolation inequalities involving  the eigenfunctions of  Laplace-Beltrami  operator that are useful in 
the analysis of the stability issue of the studied inverse coefficient problems. These inequalities have been obtained 
by quantifying the unique continuation property for  the Laplace-Beltrami  operator  through weighted
energy estimates with the aid of Carleman type inequalities.  We present an abstract framework 
for the inverse source problem in Section 4. Based on the introduced observability inequalities we provide 
uniqueness and stability inequalities of the  recovery of volume and boundary  lower order coefficients in wave and heat equations  from boundary measurements   in respectively Sections 5 and 6.

\section{Observability inequalities}

We collect in this section various observability inequalities that are necessary to the analysis of the inverse problems we want to tackle in this text. Since most of these results  are well recorded in the literature we limited ourselves to give their precise statement and provide the references where the proofs can be found.

\subsection{Wave and heat equations in a Riemannian manifold}

Let $n\geq2$ be an integer and consider $M=(M,g)$ a compact $n$-dimensional Riemannian manifold with boundary. By  a manifold with boundary we mean a $C^\infty$ manifold and its boundary is $C^\infty$ manifold of dimension $n-1$. Throughout, we adopt the Einstein convention summation for repeated indices. If in any term the same index name appears twice, as both  an upper and a lower index, that term is assumed to be summed from $1$ to $n$.

In local coordinates system $x=(x^1,\ldots ,x^n)$,
\[
g=g_{ij}dx^i\otimes dx^j.
\]

Let $(\partial _1,\ldots ,\partial _n)$ be the dual basis of $(x^1,\ldots ,x^n)$. For two vector fields $X=X^i\partial _i$ and $Y=Y^j\partial _j$ over $M$, set
\[
\langle X,Y\rangle = g_{ij}X^iY^j
\]
and  $|X|=\sqrt{\langle X,X\rangle}$.

Recall that the gradient of $u\in C^\infty (M)$ is the vector field given by
\[
\nabla u= g^{ij}\partial _iu\partial _j 
\]
and the Laplace-Beltrami operator is the operator acting as follows
\[
\Delta u=\frac{1}{\sqrt{\mbox{det}\, g}}\partial _i\left( \sqrt{\mbox{det}\, g}\, g^{ij}\partial _ju\right),
\]
where $(g^{ij})$ denote the inverse of the metric $(g_{ij})$.

We are first concerned with observability inequalities for the wave equation. Consider then the following initial-boundary value problem, abbreviated to  IBVP's in the sequel,  for the wave equation:
\begin{equation}\label{O1}
\left\{
\begin{array}{lll}
 \partial _t^2 u - \Delta u + q(x)u + a(x) \partial_t u = 0 \quad &\mbox{in}\;   M \times (0,\tau), 
 \\
u = 0 &\mbox{on}\;  \partial M \times (0,\tau), 
\\
u(\cdot ,0) = u_0,\; \partial_t u (\cdot ,0) = u_1.
\end{array}
\right.
\end{equation}

The usual energy space for the wave equation is given by
\[
\mathcal{H} =H_0^1(M) \oplus L^2(M).
\]
 According to \cite[sections 5 and 6, Chapter XVIII]{DL} or \cite[Chapter 2] {BY}),  for any $q,a\in L^\infty (M  )$, $\tau >0$ and $(u_0,u_1)\in \mathcal{H}_0$, the IBVP \eqref{O1} has a unique solution 
 \[ 
 u=u(q,a, (u_0,u_1))\in C([0,\tau ],H_0^1(M  ))\] so that $\partial _tu\in C([0,\tau ],L^2(M))$. If in addition
 \[
 \|q\|_\infty +\|a\|_\infty  \le \aleph,
 \]
 for some constant $\aleph>0$, then by the energy estimate  
\begin{equation}\label{O2}
\|u\|_{C([0,\tau ], H_0^1(M ))}+\|\partial _t u\|_{C([0,\tau ], L^2(M))}\leq C\|(u_0,u_1)\|_{\mathcal{H}_0},
\end{equation}
holds with $C=C(\aleph)>0$ is a nondecreasing function.

Denote by $\nu$ the  unit normal vector field pointing inward $M$ and set $\partial _\nu u =\langle \nabla u,\nu\rangle$. From \cite[Lemma 2.4.1]{BY} $\partial _\nu u\in L^2(\partial M \times (0,\tau))$ and
\begin{equation}\label{O3}
\|\partial_\nu u\|_{L^2(\partial M \times (0,\tau))}\le c_M\left( \|(u_0,u_1)\|_{\mathcal{H}_0} +\|qu + a\partial_t u\|_{L^1((0,\tau) , L^2(M)} \right),
\end{equation}
where $c_M$ is a constant depending only on $M$.

In light of \eqref{O2}, \eqref{O3} yields
\begin{equation}\label{O4}
\|\partial_\nu u\|_{L^2(\partial M \times (0,\tau))}\le C\|(u_0,u_1)\|_{\mathcal{H}},
\end{equation}
with a constant $C$ of the same form as in \eqref{O2}.

Let $\Gamma$ be a non empty open subset of $\partial M$ and $\tau >0$ so that  $(\Gamma ,\tau)$ geometrically control $M$. This means that every
generalized geodesic  traveling at speed one in $M$ meets $\Gamma$ in a non-diffractive point at a time $t\in (0,\tau )$ (we refer to \cite{Le} for more details).

Fix $(q_0,a_0)\in L^\infty (M)\times L^\infty (M)$. In light of \cite[theorem page 169]{Le} (which remains valid for the wave operator plus an operator involving space derivatives of first order) and bearing in mind that controllability is equivalent to observability we can state the following inequality
\begin{equation}\label{e5}
2\kappa_0 \| (u_0,u_1)\|_{\mathcal{H}} \le \| \partial _\nu u^0\|_{L^2(\Gamma \times (0,\tau ))},
\end{equation}
for some constant $\kappa_0 >0,$ where we set $u^0=u(a_0, q_0,(u_0,u_1))$ for $(u_0,u_1)\in \mathcal{H}$.

By a perturbation argument,  there exists $\beta >0$, depending on $(q_0,a_0)$ and $\kappa_0$, so that, for any $(q,a)=(q_0,a_0)+(\tilde{q},\tilde{a})$, with $(\tilde{q},\tilde{a})\in  L^\infty (M)\times L^\infty (M)$ satisfying $\|(\tilde{q},\tilde{a})\|_{ L^\infty (M)\times L^\infty (M)}\le \beta$, we have
\[
\kappa_0 \| (u_0,u_1)\|_{\mathcal{H}} \le \| \partial _\nu u\|_{L^2(\Gamma \times (0,\tau ))}.
\]

Here $\kappa _0>0$ is the same as in previous inequality and $ u=u( a, q, (u_0,u_1))$.

\begin{theorem}\label{theorem-O1}
Let   $(q_0,a_0)\in L^\infty (M)\times L^\infty (M)$ and assume  that $(\Gamma, \tau) $ geometrically control $M$. There exist $\kappa>0$ and $\beta>0$, only depending on $\Gamma, M$ and  $(q_0,a_0)$, such that for any $(q,a)=(q_0,a_0)+(\tilde{q},\tilde{a})$  with  $(\tilde{q},\tilde{a})$
satisfying 
\[
\|(\tilde{q},\tilde{a})\|_{ L^\infty (M)\times L^\infty (M)}\le \beta
\]
we have
\begin{equation}\label{O5}
\kappa \| (u_0,u_1)\|_{\mathcal{H}_0} \le \| \partial _\nu u\|_{L^2(\Gamma \times (0,\tau ))},
\end{equation}
where $u=u(q,a,(u_0,u_1))$.
\end{theorem}

Next, we examine the case where we do not assume that $(\Gamma ,\tau)$ geometrically control $M$. Define
\[
\mathbf{d}(\Gamma )=\sup\{d(x,\Gamma );\; x\in M\},
\]
let
\[
v=v(q,(u_0,v_0))=u(q,0,(u_0,u_1))
\]
and set
\[
\mathcal{H}_{-1}=L^2(M)\oplus H^{-1}(M ).
\]

In light of \cite[Corollary 3.2]{LL} we have

\begin{theorem}\label{theorem-O2}
Let $\aleph>0$. Under the assumption $\tau >2\mathbf{d}(\Gamma )$ there exist positive constants $C$, $\kappa$ and  $\epsilon _0$ so that for any $q\in L^\infty (M)$ with $\|q\|_{L^\infty (M)}\le \aleph$ we have
\begin{equation}\label{O6}
C\|(u_0,u_1)\|_{\mathcal{H}_{-1}}\le e^{\kappa \epsilon}\| \partial _\nu v\|_{L^2(\Gamma \times (0,\tau ))}+\frac{1}{\epsilon}\|(u_0,u_1)\|_{\mathcal{H}},\quad (u_0,u_1)\in \mathcal{H},\; \epsilon \ge \epsilon_0.
\end{equation}
Here $v=v(q,(u_0,v_0))$.
\end{theorem}

We now give an observability inequality for a parabolic equation. Consider then the IBVP
\begin{equation}\label{O7}
\left\{
\begin{array}{lll}
 \partial _t u - \Delta u + q(x)u = 0 \quad &\mbox{in}\;   M \times (0,\tau), 
 \\
u = 0 &\mbox{on}\;  \partial M \times (0,\tau), 
\\
u(\cdot ,0) = u_0.
\end{array}
\right.
\end{equation}

For $q\in L^\infty (M)$ let $A_q=\Delta -q$ with domain $D(A_q)=H_0^1(M)\cap H^2(M)$. As $A_0$ (that is $A_q$ with $q=0$) is an m-dissipative operator we deduce form the well established theory of continuous semigroups that $A_q$ generates a strongly continuous semigroup $e^{tA_q}$. Therefore, for any $u_0\in L^2(M)$, the IBVP has a unique solution
\[
u=u(q,u_0)=e^{tA_q}u_0\in C([0,\tau], L^2(M))\cap C^1(]0,\tau ],H^2(M)\cap H_0^1(M)).
\]
The perturbation argument  we used previously for the wave equation is in fact stated in general abstract setting 
 \cite[Proposition 6.3.3, page 189]{TW}, which is also applicable  for the heat equation. This together with \cite[Corollary 4]{LR} yield the following final time observability inequality.

\begin{theorem}\label{theorem-O3}
Let $\tau >0$, $\Gamma$ a non empty open subset of $\partial M$ and $\aleph>0$. There exists a constant $C >0$ so that for any $q\in L^\infty (M)$ satisfying $\|q\|_{L^2(M)}\le \aleph$ we have
\begin{equation}\label{O8}
\|u(\cdot ,\tau )\|_{L^2(M)}\le C\|\partial _\nu u\|_{L^2(\Gamma \times (0,\tau ))},
\end{equation}
where $u=u(q,u_0)$ with $u_0\in L^2(M)$.
\end{theorem}

\subsection{The wave equation in a rectangular domain with boundary damping}

Consider on $\Omega = (0,1)\times (0,1)$ the IBVP
\begin{equation}\label{O9}
\left\{
\begin{array}{lll}
 \partial _t^2 u - \Delta u = 0 \quad &\mbox{in}\;   \Omega \times (0,\tau), 
 \\
u = 0 &\mbox{on}\;  \Gamma _0 \times (0,\tau), 
\\
\partial _\nu u +a\partial _tu=0 &\mbox{on}\;  \Gamma _1 \times (0,\tau), 
\\
u(\cdot ,0) = u_0,\; \partial_t u (\cdot ,0) = u_1.
\end{array}
\right.
\end{equation}
Here 
\begin{align*}
&\Gamma _0=((0,1)\times \{1\})\cup (\{1\}\times (0,1)),
\\
&\Gamma _1=((0,1)\times \{0\})\cup (\{0\}\times (0,1))
\end{align*}
and $\partial _\nu =\nu \cdot \nabla$ is the derivative along $\nu$, the unit normal vector pointing outward of $\Omega$. Note that $\nu$ is everywhere defined except at the vertices of $\Omega$.

We identify in the sequel $a|_{(0,1)\times \{0\}}$ by $a_1=a_1(x)$, $x\in (0,1)$ and $a|_{\{0\}\times (0,1)}$ by $a_2=a_2(y)$, $y\in (0,1)$. In that case it is natural to identify $a$, defined on $\Gamma _1$, by the pair $(a_1,a_2)$.

Fix $1/2<\alpha \le 1$ and let
\[
\mathscr{A}=\{ b=(b_1,b_2)\in C^\alpha ([0,1])\oplus C^\alpha ([0,1]),\; b_1(0)=b_2(0),\; b_j\geq 0\}. 
\]

Let $V=\{ u\in H^1(\Omega );\; u=0\; \textrm{on}\; \Gamma _0\}$ and define on $V\oplus L^2(\Omega )$ the unbounded operator $A_a$, $a\in \mathscr{A}$,  by
\[
A_a= (w,\Delta v),\quad D(A_a)=\{ (v,w)\in V\oplus V;\; \Delta v\in L^2(\Omega )\; \textrm{and}\; \partial _\nu v=-aw\; \textrm{on}\; \Gamma _1\}.
\]
From \cite{AC3} $A_a$ generates a strongly continuous semigroup $e^{tA_a}$. Whence, for any $(u_0,u_1)\in D(A_a)$, the IBVP \eqref{O9} has a solution $u=u(a,(u_0,u_1))$ so that
\[
(u,\partial _tu)\in C([0,\tau ],D(A_a))\cap C^1([0,\tau ],V\oplus L^2(\Omega )).
\]

We proved in \cite[Corollary 2.2]{AC3} the following observability inequality
\begin{theorem}\label{theorem-O4}
Fix $0<\delta _0 <\delta _1$. Then there exist $\tau _0>0$ and $\kappa>0$, depending only on $\delta _0$ and $\delta_1$, so that for any $\tau \geq \tau _0$ and $a\in \mathscr{A}$ satisfying  $\delta _0 \le a \le \delta _1$  on $\Gamma _1$ we have
\[
 \kappa \| (u_0,u_1)\|_{V\oplus L^2(\Omega )}\leq \|\partial _\nu u\|_{L^2(\Gamma _1\times (0,\tau))},
\]
where $u=u(a,(u_0,u_1))$, with $(u_0,u_1)\in D(A_a)$.
\end{theorem}

It is worth noticing that $\Gamma_1$ satisfies the geometric control condition given in the multiplier method. We also point out that a special case was considered by the third author and Ren \cite{RT} in which the observation is made only on one side of $\Gamma_1$.

\section{Weighted interpolation inequalities}

We aim in the present section establishing two weighted interpolation inequalities. These inequalities will be useful in the proof of H\"older stability estimates for certain inverse problems we discuss in the coming sections.

As in the preceding section $M$ is a compact $n$-dimensional Riemannian manifold with boundary. 

Consider the Hardy's inequality 
\begin{equation}\label{wii1}
\int_M|\nabla f(x)|^2dV \geq c\int_M \frac{|f(x)|^2}{d(x,\partial M)^2}dV,\quad f\in H_0^1(M),
\end{equation}
for some constant $c>0$, where $dV$ is the volume form on $M$, $d$ is the geodesic distance introduced 
previously and $d(\cdot ,\partial M)$ is the distance to $\partial M$.

Define $r_x(v)=\inf \{|t|;\; \gamma_{x,v}(t)\not\in M\}$, where $\gamma_{x,v}$ is the geodesic satisfying the initial condition $\gamma_{x,v}=x$ and $\dot{\gamma}_{x,v}=v$. It was observed in \cite{Ra} that Hardy's inequality \eqref{wii1} holds for any open subset $\mathcal{O}$, of a complete Riemannian manifold, whenever $\mathcal{O}$ has the following uniform interior cone property: there are an angle $\alpha >0$ and a constant $c_0>0$ so that, for any $x\in M$, there exists an $\alpha$-angled cone $C_x\subset T_xM$ $^[$\footnote{Here  $C_x=\{ Y\in T_xM;\; \angle \langle X,Y\rangle <\alpha \}$, for some $X\in T_xM$.}$^]$ with the property that $r_x(v)\le c_0d(x,\partial M)$, for all $v\in C_x$. The proof of this result follows the method by Davies \cite[page 25]{Da} for the flat case. Since in our case $M$ is a compact Riemannian manifold, it is obvious that it satisfies the uniform interior cone property. Then slight modifications of the proof in \cite{Ra} show that Hardy's inequality is satisfied for any compact Riemannian manifold.

It is worth mentioning that Hardy's inequality holds for any bounded Lipschitz domain of $\mathbb{R}^n$ with constant $c\le 1/4$, with equality if and only if $\Omega$ is convex.

The following Hopf's maximum principle is a key ingredient in establishing our first weighted interpolation inequality.

\begin{lemma}\label{lemma-wii1}
Let $q\in C(M)$ and  $u\in C^2(M)\cap H_0^1(M)$ satisfying $q\le 0$ and $\Delta u+qu\le 0$. If $u$ is non identically equal to zero then $u>0$ in $M$ and $\partial _\nu u(y)=\langle \nabla u(y),\nu (y)\rangle >0$ for any $y\in \partial M$.
\end{lemma}

\begin{proof}
Similar to that of \cite[Lemma 3.4, page 34 and Theorem 3.5, page 35]{GT}. The tangent ball in the classical Hopf's lemma is substitute by a tangent geodesic ball (see the construction in \cite[Proof of Theorem 9.2, page 51]{PS}).
\end{proof}

\begin{proposition}\label{proposition-wii1}
Let $q\in C(M)$ and  $u\in C^2(M)\cap H_0^1(M)$ satisfying $q\le 0$ and $\Delta u+qu\le 0$.
If $u$ is  non identically equal to zero then
\[
u(x)\ge c_ud(x,\partial M),\quad x\in M,
\]
where the constant $c_u$ only depends on $u$ and $M$.
\end{proposition}

\begin{proof}
Let $0<\epsilon $ to be specified later. Let $x\in M$ so that $d(x,\partial M)\le \epsilon$ and $y\in \partial M$ satisfying $d(x,\partial M )=d(x,y)$. Since $M$ is complete there exist a unit speed minimizing geodesic $\gamma : [0,r]\rightarrow M$ such that $\gamma (0)=y$, $\gamma (r)=x$ and $\dot{\gamma}(0)=\nu (y)$, where we set $r=d(x,\partial M)$ (see for instance \cite[page 150]{Pe}).

Define $\phi(t)=u(\gamma (t))$. Then
\begin{align*}
&\phi '(t)=du(\gamma (t))(\dot{\gamma}(t))
\\
&\phi''(t)=d^2u(\gamma (t))(\dot{\gamma}(t),\dot{\gamma}(t))+du(\gamma (t))(\ddot{\gamma}(t)).
\end{align*}
Here $\dot{\gamma}(t)=\dot{\gamma}^i(t)\partial _i\in T_{\gamma (t)}$. Observe that by the geodesic equation
\[
\ddot{\gamma}^k(t)=-\dot{\gamma}^i(t)\dot{\gamma}^j(t)\Gamma_{ij}^k(\gamma (t)),
\]
where $\Gamma_{ij}^k$ are the Christoffel symbols associated to the metric $g$.

We get by taking into account that $\phi '(0)=du(y)(\nu (y))=\langle \nabla u(y),\nu (y)\rangle =\partial _\nu u(y)$
\[
\phi (r)=r\partial _\nu u(y)+\frac{r^2}{2}\phi ''(st),
\]
for some $0<s<1$. Hence there exists $c>0$ depending on $u$ and $M$ so that
\[
\phi(r)\ge 2r\eta -cr^2\ge r\eta + r(\eta -c\epsilon),
\]
with $2\eta = \min_{y\in\Gamma} \partial _\nu u(y) >0$ (by the compactness of $\partial M$ and Lemma \ref{lemma-wii1}). Thus
\[
\phi (r)\ge r\eta 
\]
provided that $\epsilon \le \eta /c$. In other words we proved
\begin{equation}\label{wii2}
u(x)=\phi (r )\ge r \eta =\eta d(x,\partial M ).
\end{equation}

On the other hand an elementary compactness argument yields, where  $M^\epsilon =\{x\in M;\; d(x,\partial M )\ge \epsilon\}$, 
\begin{equation}\label{wii3}
u(x)\ge \min_{z\in M^\epsilon}u(z)\ge \frac{\min_{z\in M^\epsilon}u(z)}{\max_{z\in M^\epsilon}d(z,\partial M )}d(x,\partial M),\quad x\in M^\epsilon .
\end{equation}
In light of \eqref{wii2} and \eqref{wii3} we end up getting
\[
u(x)\ge c_u d(x,\partial M ),\quad x\in M.
\]
The proof is then complete.
\end{proof}

A consequence of Proposition \ref{proposition-wii1} is the following corollary.

\begin{corollary}\label{corollary-wii1}
Let $q\in C(M)$, $q\le 0$, and $u\in C^2(M) \cap H_0^1(M)$ non identically equal to zero satisfying $\Delta u+qu\le 0$. There exists a constant $C_u$ only depending on $u$ and $M$ so that we have
\[
\|f\|_{L^2(M)}\le C_u\|fu\|_{L^2(M)}^{\frac{1}{2}}\|f\|_{H^2(M)}^{\frac{1}{2}}
\]
for any $f\in H^2(M)$.
\end{corollary}

\begin{proof}
By Proposition \ref{proposition-wii1} $u(x)\ge c_ud(x,\partial M )$. Therefore
\[
\int_{M} f(x)^2dV(x)\le c_u^{-1} \int_{M}\frac{f(x)^2u(x)^2}{d(x,\partial M)^2}dV(x).
\]
Combined with Hardy's inequality \eqref{wii1} this estimate gives
\begin{equation}\label{wii4}
\int_{M} f(x)^2dV(x)\le c_u^{-1}c \int_{M}|\nabla (fu)(x)|^2dV(x).
\end{equation}
But from usual interpolation inequalities we have
\[
\|fu\|_{H^1(M)}\le C\| fu\|_{L^2(M)}^{\frac{1}{2}}\|fu\|_{H^2(M)}^{\frac{1}{2}},
\]
where the constant $C$ only depends on $M$.

Whence \eqref{wii4} implies
\[
\|f\|_{L^2(M)}\le C_u\|fu\|_{L^2(M)}^{\frac{1}{2}}\|f\|_{H^2(M)}^{\frac{1}{2}},
\]
which is the expected inequality
\end{proof}

Let $0\le q\in C^1 (M)$ be fixed and consider the operator $A=-\Delta+q$ with domain $D(A)=H^2(M) \cap H_0^1(M)$. An extension of \cite[Theorem 8.38, page 214]{GT} to a compact Riemannian manifold with boundary shows that the first eigenvalue of $A$, denoted by $\lambda _1$, is simple and has a positive eigenfunction. Let then $\phi _1\in C^2 (M)$ (by elliptic regularity) be the unique first eigenfunction satisfying $\phi_1 >0$ and normalized by $\|\phi _1\|_{L^2(M)}=1$. Since $\Delta \phi _1-q\phi_1 =-\lambda _1\phi_1$ the Hopf's maximum principle is applicable to $\phi_1$. Therefore a particular weight in the preceding corollary is obtained by taking $u=\phi _1$.

\begin{corollary}\label{corollary-wii2}
There exists a constant $c>0$, depending on $\phi_1$, so that  we have
\begin{equation}\label{wii5}
\|f\|_{L^2(M)}\le c\|f\phi _1\|_{L^2(M)}^{\frac{1}{2}}\|f\|_{H^2(M)}^{\frac{1}{2}}
\end{equation}
for any $f\in H^2(M)$.
\end{corollary}

The second weighted interpolation inequality relies on the following proposition.

\begin{proposition}\label{proposition-wii2}
Let $p\in L^\infty (M)$ and $u\in W^{2,n}(M)$ satisfying $(\Delta +p)\varphi =0$ in $M$ and $\varphi^2\in W^{2,n}(M)$. Then there exists $\delta =\delta (\varphi ) >0$ so that $|\varphi| ^{-\delta} \in L^1(M)$. 
\end{proposition}

 It is worth mentioning that in general $\delta <1$ as soon as $\varphi$ vanishes at some point $x_0\in M$. Consider for instance in the flat case $\psi (x)\sim |x-x_0|^k$ near $x_0$ if $x_0$ is a zero of order $k$. It is then clear that $|\varphi|^{-\delta}$ is locally integrable in a neighborhood of $x_0$ if and only if $\delta k<n-1$. In consequence $\delta <1$ whenever $k\ge n-1$.

\begin{proof}[Sketch of the proof]

{\it First step.} Denote by $B$ the unit ball of $\mathbb{R}^n$ and let $B_+=B\cap \mathbb{R}_+^n$, with $\mathbb{R}_+^n=\{ x=(x',x_n)\in \mathbb{R}^n;\; x_n >0\}$. Let $L$ be a second order differential operator acting as follows
\[
Lu=\partial _j(a_{ij}\partial _i u )+ V\cdot \nabla u+du.
\]
Assume that $(a_{ij})$ is a symmetric matrix with entries in
$C^1(2\overline{B_+} )$, $V\in L^\infty (2B_+)^n$ is real
valued 
and $d\in L^\infty (2B_+)$ is complex valued. Suppose furthermore that
\[
a_{ij}(x)\xi _j\cdot \xi_j\ge \kappa_0 |\xi |^2,\quad x\in 2B_+,\; \xi \in \mathbb{R}^n,
\]
for some $\kappa_0 >0$.

Let $u\in W^{2,n}(2B_+)\cap C^0(2\overline{B_+})$ be a weak solution of $Lu=0$
satisfying $u=0$ on $\partial (2B_+)\cap \overline{\mathbb{R}^n_+}$ and $|u|^2\in W^{2,n}(2B_+)\cap C^0(2\overline{B_+})$.
\smallskip
From \cite[Theorem 1.1, page 942]{AE} there exists a constant $C$, depending
of $u$, so that the following doubling inequality at the boundary
\begin{equation*}
\int_{B_{2r}\cap B_+}{|u|}^2dx \le C\int_{B_r\cap B_+}{|u|}^2dx,
\end{equation*}
holds for any ball $B_{2r}$ of radius $2r$ contained in $2B$.

On the other hand  simple calculations yield, where $v= \Re u$ and $w=\Im u$,
\begin{align*}
\partial _j(a_{ij}\partial _i |u|^2 )+ 2V\cdot \nabla |u|^2&+4(|\Re d|+|\Im d|)|u|^2
\\
&\ge 2a_{ij}\partial_iv\partial_jv+2a_{ij}\partial_iw\partial_jw \ge 0\quad \mbox{in}\; 2B_+
\end{align*}
and $|u|^2=0$ on $\partial (2B_+)\cap
\overline{\mathbb{R}^n_+}$.

Harnak's inequality at the boundary (see \cite[Theorem 9.26, page 250]{GT}) then yields 
\begin{equation*}
\sup_{B_r \cap B_+}{|u|}^2\le \frac{C}{|B_{2r}|}\int_{B_{2r}\cap B_+}
{|u|}^2dx,
\end{equation*}
for any ball $B_{2r}$ of radius $2r$ contained in $2B$.

Define $\tilde{u}$  by
\begin{align*}
&\tilde{u} (x',x_n)=u(x',x_n)\quad \mbox{if}\; 
(x',x_n)\in 2B_+ ,
\\ 
&\tilde{u} (x',x_n)=u(x',-x_n)\quad \mbox{if}\; (x',-x_n)\in 2B_+.
\end{align*}
Therefore  $\tilde{u} $ belongs  to $H^1(2B)\cap L^\infty(2B)$ and satisfies 
\begin{eqnarray}
\int_{B_{2r}}|\tilde{u}|^2dx \le C\int_{B_r}|\tilde{u}|^2dx,\label{wii6}
\\
\sup_{B_r}|\tilde{u}|^2\le \frac{C}{|B_{2r}|}\int_{B_{2r}}
{|\tilde u|}^2dx,\label{wii7}
\end{eqnarray}
 for any ball $B_{2r}$ of radius $2r$ contained in $2B$.

Inequalities \eqref{wii6} and \eqref{wii7} at hand we mimic the proof of \cite[Theorem 4.2, page 1784]{CT} in order to obtain that $|\tilde{u}|^{-\delta}\in L^1(B)$, for some $\delta >0$ depending on $u$. Whence $|u|^{-\delta}\in L^1(B_+)$.

{\it Second step.} As $\partial M$ is compact there exists a finite cover $(U_\alpha )$ of $\partial M$  and  $C^\infty$-diffeomorphisms  $f_\alpha: U_\alpha \rightarrow 2B$ so that $f_\alpha (U_\alpha \cap M)=2B_+$, $f_\alpha (U_\alpha \cap \partial M)=2B \cap \overline{\mathbb{R}^n_+}$ and, for any $x\in \partial M$, $x\in V_\alpha =f_\alpha^{-1}(B)$, for some $\alpha$. Then $u_\alpha =\varphi \circ f_\alpha ^{-1}$ satisfies $L_\alpha u_\alpha =0$ in $2B$ and $u=0$ on $\partial (2B_+)\cap \overline{\mathbb{R}^n_+}$ for some $L=L_\alpha$ satisfying the conditions of the first step. Hence $|u_\alpha|^{-\delta _\alpha}\in L^1(B_+)$ and then $|\varphi|^{-\delta _\alpha}\in L^1(V_\alpha )$. Let $V$ the union of $V_\alpha$'s. Since $u\in L^\infty (V)$ we get $|u|^{-\delta_0}\in L^1(V)$ with $\delta_0=\min \delta_\alpha$. Next, let $\epsilon$ sufficiently small in such a way that $M\setminus M_\epsilon\subset V$, where $M_\epsilon =\{ x\in M;\; \mbox{dist}(x,\partial M)>\epsilon \}$. Proceeding as previously  it is not hard to get that there exists $\delta _1$ so that $|\varphi|^{-\delta _1}\in L^1(M_{\epsilon/2})$. Finally, as it is expected we obtain that $|\varphi|^{-\delta}\in L^1(M)$ with $\delta =\min (\delta_0,\delta_1)$.
\end{proof}

\begin{lemma}\label{lemma-wii2} Let $\varphi$ be as in Proposition \ref{proposition-wii2}.
There exists a constant $C>0$, depending on $\varphi$,  so that we have
\[
\|f\|_{L^2(M)}\le C\|f\|_{L^\infty (M)}^{\frac{2}{2+\delta}}\|f\varphi\|_{L^2(M)}^{\frac{\delta}{2+\delta}},
\]
for any $f\in L^\infty (M)$
\end{lemma}

\begin{proof}
Let $\delta=\delta (\varphi )$ given as in the preceding proposition. Since $\varphi$  belongs to $L^\infty (M)$, substituting $\delta$ by $\min (1,\delta )$ if necessary, we may assume that $\delta <2$. We get by applying Cauchy-Schwarz's inequality 
\[
\int_{M}|f|^{\delta /2} dV \le \||f\varphi |^\delta \|_{L^1(M)}^{1/2}\||\varphi| ^{-\delta}\|^{1/2}_{L^1(M)}.
\]
But by H\"older's inequality
\[
\||f\varphi |^\delta \|_{L^1(M)}^{1/2}\le \mbox{Vol}(M)^{(2-\delta )/4}\|f\varphi\|_{L^2(M)}^{\delta /2}.
\]
Whence
\begin{equation}\label{wii8}
\||f|^{\delta /2}\|_{L^1 (M)}\le \mbox{Vol}(M)^{(2-\delta )/4}\|f\varphi\|_{L^2(M)}^{\delta /2}\||\varphi| ^{-\delta}\|^{1/2}_{L^1(M)}.
\end{equation}
On the other hand
\begin{equation}\label{wii9}
\|f\|_{L^2(M)}\le \|f\|_{L^\infty (M)}^{1-\delta/4}\||f|^{\delta /2}\|_{L^1 (M)}^{1/2}.
\end{equation}
A combination of \eqref{wii8} and \eqref{wii9} yields
\[
\|f\|_{L^2(M)}\le C\|f\|_{L^\infty (M)}^{1-\delta/4}\|f\varphi\|_{L^2(M)}^{\delta/4},
\]
which is the expected inequality.
\end{proof}

\section{Inverse source problem: abstract framework}

Let $H$ be a Hilbert space  and $A :D(A)  \subset H \rightarrow H$ be the generator of continuous semigroup $T(t)$. An operator  $\mathscr{C} \in \mathscr{B}(D(A),Y)$,  $Y$ is another Hilbert space which is identified with its dual space, is called an admissible observation  for $T(t)$ if for some (and hence for all) $\tau >0$ the operator $\Psi \in \mathscr{B}(D(A),L^2((0,\tau ),Y))$ given by
\[
(\Psi  x)(t)=\mathscr{C}T(t)x,\;\; t\in [0,\tau ],\quad x\in D(A),
\]
has a bounded extension to $H$.

We  introduce the definition of exact observability for the system
\begin{align}\label{isp1}
&z'(t)=Az(t),\quad z(0)=x,
\\
&y(t)=\mathscr{C}z(t),\label{isp2}
\end{align}
where $\mathscr{C}$ is an admissible observation  for $T(t)$. The pair $(A,\mathscr{C})$ is said exactly observable at time $\tau >0$ if there is a constant $\kappa $ such that the solution $(z,y)$ of \eqref{isp1} and \eqref{isp2} satisfies 
\[
\int_0^\tau \|y(t)\|_Y^2dt\geq \kappa ^2 \|x\|_H^2,\;\; x\in D(A).
\]
Or equivalently
\begin{equation}\label{isp3}
\int_0^\tau \|(\Psi  x)(t)\|_Y^2dt\geq \kappa ^2 \|x\|_H^2,\;\; x\in D(A).
\end{equation}

Consider the Cauchy problem
\begin{equation}\label{isp4}
z'(t)=Az(t)+\lambda (t)x,\;\; z(0)=0
\end{equation}
and set
\begin{equation}\label{isp5}
y(t)=\mathscr{C}z(t),\;\; t\in [0,\tau ].
\end{equation}

By Duhamel's formula we have
\begin{equation}\label{isp6}
y(t)=\int_0^t \lambda (t-s)\mathscr{C}T(s)xds=\int_0^t\lambda (t-s)(\Psi  x)(s)ds.
\end{equation}

Let
\[
H^1_\ell ((0,\tau), Y) = \left\{u \in H^1((0,\tau), Y); \; u(0) = 0 \right\}.
\]

Define the operator $S:L^2((0,\tau), Y)\longrightarrow H^1_\ell ((0,\tau ) ,Y)$ by
\begin{equation}\label{isp7}
(Sh)(t)=\int_0^t\lambda (t-s)h(s)ds.
\end{equation}

If $E =S\Psi$ then \eqref{isp6} takes the form
\[
y(t)=(E x)(t).
\]

\begin{theorem}\label{theorem-isp1}
Assume that $(A,\mathscr{C})$ is exactly observable for $\tau \geq \tau_0$, for some $\tau_0>0$. Let $\lambda \in H^1(0,\tau )$ satisfies $\lambda (0)\ne 0$. Then $E$ is one-to-one from $H$ onto $H^1_\ell ((0,\tau), Y)$ and
\begin{equation}\label{isp8}
\frac{\kappa |\lambda (0)|}{\sqrt{2}}e^{-\tau \frac{\|\lambda '\|^2_{L^2(0,\tau)}}{|\lambda (0)|^2 }}\|x\|_H\leq  \|Ex\|_{H^1_\ell ((0,\tau), Y)},\quad x\in H.
\end{equation}
\end{theorem}

\begin{proof}
Taking first the derivative with respect to $t$ of both sides of the integral equation
\[
\int_0^t \lambda (t-s)\varphi (s) ds=\psi (t)
\]
we get the following Volterra integral equation of second kind
\[
\lambda (0)\varphi (t) +\int_0^t\lambda '(t-s)\varphi (s)ds=\psi '(t).
\]
Mimicking the proof of  \cite[Theorem 2, page 33]{Ho} we obtain that this integral equation has a unique solution $\varphi \in L^2((0,\tau ) ,Y)$ and
\begin{align*}
\|\varphi \|_{L^2((0,\tau ),Y)}&\leq C \|\psi '\|_{L^2((0,\tau ),Y)}
\\ 
&\le C \|\psi \|_{H^1_\ell ((0,\tau ),Y)}.
\end{align*}
Here $C =C(\lambda )$ is a constant. 

For estimating the constant $C$ above we first use the elementary convexity inequality $(a+b)^2\leq 2(a^2+b^2)$ in order to get
\[
\| |\lambda (0)|\varphi (t)\|_Y^2\leq 2\left( \int_0^t\frac{|\lambda '(t-s)}{|\lambda (0)|}\left[|\lambda (0)|\| \varphi (s)\|_Y\right]ds \right)^2+2\|\psi '(t)\|_Y^2.
\]
Thus
\[
|\lambda (0)|^2\| \varphi (t)\|_Y^2\leq 2\frac{\|\lambda '\|_{L^2((0,\tau))}^2}{|\lambda (0)|^2}\int_0^t|\varphi (0)|^2\| \varphi (s)\|_Y^2ds +2\|\psi '(t)\|_Y^2
\]
by the Cauchy-Schwarz's inequality. Therefore using  Gronwall's lemma we obtain in a straightforward manner
\[
\| \varphi \|_{L^2((0,\tau ),Y)}\leq \frac{\sqrt{2}}{|\lambda (0)|}e^{\tau \frac{\|\lambda '\|_{L^2((0,\tau))}^2}{|\lambda (0)|^2}}\|\psi '\|_{L^2((0,\tau ),Y)}
\]
and then
\[
\| \varphi \|_{L^2((0,\tau ),Y)}\leq \frac{\sqrt{2}}{|\lambda (0)|}e^{\tau\frac{\|\lambda '\|_{L^2((0,\tau))}^2}{|\lambda (0)|^2}}\|S\varphi \|_{H^1_\ell ((0,\tau ),Y)}.
\]
In light of \eqref{isp3} we end up getting
\[
\|Ex\|_{H^1_\ell ((0,\tau), Y)}\geq \frac{\kappa |\lambda (0)|}{\sqrt{2}}e^{-\tau \frac{\|\lambda '\|^2_{L^2((0,\tau))}}{|\lambda (0)|^2 }} \|x\|_H.
\]
This is the expected inequality.
\end{proof}

We shall need a variant of Theorem \ref{theorem-isp1}. If $(A,\mathscr{C})$ is as in Theorem \ref{theorem-isp1} then, as in the preceding section, by the perturbation argument in \cite[Proposition 6.3.3, page 189]{TW}, there exist $\aleph>0$ and $\kappa >0$ such that for any $P\in \mathscr{B}(H)$ satisfying $\|P\|\leq \aleph$ we have that $(A+P,\mathscr{C})$ is exactly observable with $\kappa  (P+A)\ge \kappa$.

Define $E^P$ similarly to $E$ by substituting in $E$  $A$ by $A+P$. 

\begin{theorem}\label{theorem-isp2}
Assume that $(A,\mathscr{C})$ is exactly observable for $\tau \geq \tau_0$, for some $\tau_0>0$, and let $\lambda \in H^1(0,\tau )$ satisfies $\lambda (0)\ne 0$. There exist $\aleph>0$ and $\kappa >0$ so that for any $P\in \mathscr{B}(H)$ satisfying $\|P\|\leq \aleph$ we have that  $E^P$ is one-to-one from $H$ onto $H^1_\ell ((0,\tau), Y)$ and
\begin{equation}\label{isp9}
\frac{\kappa |\lambda (0)|}{\sqrt{2}}e^{-\tau \frac{\|\lambda '\|^2_{L^2(0,\tau)}}{|\lambda (0)|^2 }}\|x\|_H\leq  \|E ^Px\|_{H^1_\ell ((0,\tau), Y)},\quad x\in H.
\end{equation}
\end{theorem}

We will consider inverse source problems with singular sources. For this purpose we need to extend Theorem \ref{theorem-isp1}. Fix then $\varrho$ in the resolvent set of $A$. Let $H_1$ be the space $D(A)$ equipped with the norm $\|x\|_1=\|(\varrho -A)x\|$ and denote by $H_{-1}$  the completion of $H$ with respect to the norm $\|x\|_{-1}=\| (\varrho -A)^{-1}x\|$. As we observed in \cite[Proposition 4.2, page 1644]{TW} and its proof,  when $x\in H_{-1}$ (which is the dual space of $H_1$ with respect to the pivot space $H$) and $\lambda \in H^1(0,\tau )$, then according to the classical extrapolation theory of semigroups  the Cauchy problem \eqref{isp1} has a unique solution $z\in C([0,\tau ],H)$. In addition $y$ given in \eqref{isp2} belongs to $L^2((0,\tau ) ,Y)$.

If $x\in H$ we have by Duhamel's formula
\begin{equation}\label{isp10}
y(t)=\int_0^t  \lambda (t-s)\mathscr{C}T(s)xds=\int_0^t \lambda (t-s)(\Psi  x)(s)ds.
\end{equation}

Let
\[
H^1_\ell ((0,\tau), Y) = \left\{u \in H^1((0,\tau), Y); \; u(0) = 0 \right\}.
\]

We define the operator $S :L^2((0,\tau), Y)\longrightarrow H^1_\ell ((0,\tau ) ,Y)$ by
\begin{equation}\label{isp11}
(S h)(t)=\int_0^t \lambda (t-s)h(s)ds.
\end{equation}

Hence $E  =S \Psi$ then \eqref{isp10} the form
\[
y(t)=(E x)(t).
\]

Let $\mathcal{Z}=(\varrho -A^\ast)^{-1}(X+\mathscr{C}^\ast Y)$.

\begin{theorem}\label{theorem-isp3} 
Assume that $(A,\mathscr{C})$ is exactly observable at time $\tau$. Then 
\\
(i) $E$ is one-to-one from $H$ onto $H^1_\ell ((0,\tau), Y)$.
\\
(ii) $E$ is extended to an isomorphism, denoted by $\tilde{E}$, from $\mathcal{Z}'$ onto $L^2((0,\tau ),Y)$.
\\
(iii) There exists a constant $\tilde{\kappa}$, independent of $\lambda$, so that
\begin{equation}\label{isp12}
 \|x\|_{\mathcal{Z}'}\le \tilde{\kappa}|\lambda (0)|e^{\frac{\|\lambda '\|^2_{L^2((0,\tau ))}}{|\lambda (0)|^2}\tau}\|\tilde{E}x\|_{L^2 ((0,\tau), Y)}.
\end{equation}
\end{theorem}

\begin{proof}
We give the proof of (ii) and (iii) and we note that (i)  is contained in Theorem \ref{isp1}. We first observe that $S^\ast$, the adjoint of $S$, maps $L^2((0,\tau ),Y)$ into $H_r^1((0,\tau ),Y)$, where
\[
H_r^1((0,\tau ),Y)=\left\{u \in H^1((0,\tau), Y); \; u(\tau ) = 0 \right\}.
\]
Moreover 
\[
S^\ast h(t)=\int_t^\tau \lambda (s-t)h(s)ds,\quad h\in L^2((0,\tau ),Y).
\]

Fix $h\in L^2((0,\tau ),Y)$ and set $k=S^\ast h$. Then
\[
k'(t)= \lambda (0)h(t)-\int_t^\tau  \lambda '(s-t)h(s)ds.
\]
Hence
\begin{align*}
|\lambda (0)\|h(t)\|^2&\le \left( \int_t^\tau \frac{|\lambda '(s-t)|}{|\lambda (0)|}[ |\lambda (0)|\|h(s)\|]ds +\|k'(t)\|\right)^2
\\
&\le 2\left( \int_t^\tau \frac{|\lambda '(s-t)|}{|\lambda (0)|}[ |\lambda (0)|\|h(s)\|]ds\right)^2 +2\|k'(t)\|^2
\\
&\le 2\frac{\|\lambda '\|_{L^2((0,\tau ))}^2}{|\lambda (0)|^2}\int_0^t [|\lambda (0)|\|h(s)\|]^2ds+2\|k'(t)\|^2.
\end{align*}
The last estimate is obtained  by applying Cauchy-Schwarz's inequality.

Then Gronwall's lemma yields
\[
[|\lambda (0)|\|h(t)\|]^2\le 2e^{2\frac{\|\lambda '\|_{L^2((0,\tau ))}^2}{|\lambda (0)|^2}\tau }\|k'(t)\|^2.
\]
Therefore
\[
\|h\|_{L^2((0,\tau) ,Y)}\leq \frac{\sqrt{2}}{|\lambda (0)|}e^{\frac{\|\lambda '\|_{L^2((0,\tau ))}^2}{|\lambda (0)|^2}\tau }\|k'\|_{L^2((0,\tau) ,Y)}.
\]
Whence
\begin{equation}\label{isp13}
\|h\|_{L^2((0,\tau) ,Y)}\le \frac{\sqrt{2}}{|\lambda (0)|}e^{\frac{\|\lambda '\|_{L^2((0,\tau ))}^2}{|\lambda (0)|^2}\tau }\|S^\ast h\|_{H^1_r((0,\tau) ,Y)}.
\end{equation}
The adjoint operator of $S^\ast$, acting as a bounded operator from $[H^1_r((0,\tau);Y)]'$ into $L^2((0,\tau );Y)$, gives an extension of $S$. We denote by $\tilde{S}$ this operator. By \cite[Proposition 4.1, page 1644]{TW} $\tilde{S}$ defines an isomorphism from $[H_r((0,1);Y)]'$ onto $L^2((0,\tau );Y)$. In light of the identity
\[
\|\tilde{S}\|_{\mathscr{B}([H^1_r((0,\tau);Y)]';L^2((0,\tau ),Y))}=\|S^\ast \|_{\mathscr{B}(L^2((0,\tau );Y);H^1_r((0,\tau),Y))},
\]
\eqref{isp13} implies
\begin{equation}\label{isp14}
\frac{|\lambda (0)|}{\sqrt{2}}e^{-\frac{\|\lambda '\|_{L^2((0,\tau ))}^2}{|\lambda (0)|^2}\tau }\le \|\tilde{S}\|_{\mathscr{B}([H^1_r((0,\tau);Y)]';L^2((0,\tau );Y))}.
\end{equation}
On the other hand according to \cite[Proposition 2.13, page 1641]{TW} $\Psi $ possesses a unique bounded extension, denoted by $\tilde{\Psi}$, from $\mathcal{Z}'$ into $[H^1_r((0,\tau);Y)]'$ and there exists a constant $c>0$ so that
\begin{equation}\label{isp15}
\|\tilde{\Psi}\|_{\mathcal{B}(\mathcal{Z}';[H^1_r((0,\tau);Y)]')}\geq c.
\end{equation}
The operator $\tilde{E}=\tilde{S}\tilde{\Psi}$ gives the unique extension of $E$ to an isomorphism from $\mathcal{Z}'$ onto $L^2((0,\tau ),Y)$. 

We end up the proof by noting that \eqref{isp12} follows from \eqref{isp14} and \eqref{isp15}.
\end{proof}

\section{Inverse  problems for evolution equations associated to Laplace-Beltrami operator}

Throughout this section $M$ is a compact $n$-dimensional Riemannian manifold with boundary, $\tau >0$ and $\Gamma$ is a nonempty open subset of $\partial M$.

\subsection{Inverse source problem for the wave equation}

Consider the  IBVP  for the wave equation
\begin{equation}\label{w1.0}
\left\{
\begin{array}{lll}
 \partial _t^2 u - \Delta u + q(x)u + a(x) \partial_t u = g(t)f(x) \quad &\mbox{in}\;   M \times (0,\tau), 
 \\
u = 0 &\mbox{on}\;  \partial M \times (0,\tau), 
\\
u(\cdot ,0) = 0,\; \partial_t u (\cdot ,0) = 0.
\end{array}
\right.
\end{equation}

Assume that $(\Gamma ,\tau)$ geometrically control $M$. Fix $(q_0,a_0)\in L^{\infty}(M)\oplus L^\infty(M)$ and denote by $2\kappa$ the observability constant for $(q_0,a_0)$. In light of Theorem \ref{theorem-O1} there exists a constant $\beta>0$ only depending on $\Gamma$,  $M$ and  $(q_0,a_0)$ such that, for any $(q,a)=(q_0,a_0)+(\tilde{q},\tilde{a})$  with  $(\tilde{q},\tilde{a})\in L^{\infty}(M)\oplus L^\infty(M)$ satisfying 
\begin{equation}\label{uo}
\|(\tilde{q},\tilde{a})\|_{ L^\infty (M)\times L^\infty (M)}\le \beta,
\end{equation}
the observability contant for $(q,a)$ is $\kappa$. We denote the set of couples $(q,a)\in L^{\infty}(M)\oplus L^\infty(M)$ of the form $(q,a)=(q_0,a_0)+(\tilde{q},\tilde{a})$, where $(\tilde{q},\tilde{a})\in L^{\infty}(M)\oplus L^\infty(M)$ satisfies \eqref{uo}, by $\mathscr{D}$. 

Let $f\in L^2(M)$ and $g\in H^1(0,\tau )$ with $g(0)\ne 0$. We have according to Theorem \ref{theorem-isp1}
\begin{equation}\label{w0}
\|f\|_{L^2(M)}\le \frac{\sqrt{2}}{\kappa |g(0)|}e^{\tau\frac{\|g'\|_{L^2(0,\tau)}^2}{|g(0)|^2}}\|\partial _\nu u\|_{H^1((0,\tau ),L^2(\Gamma ))},
\end{equation}
 where $u=u(q,a,f,g)$ denotes the solution of the IBVP \eqref{w1.0}.

An immediate consequence of this inequality is the following theorem.

\begin{theorem}\label{theorem-w1.0}
Assume that $(\Gamma ,\tau )$ geometrically control $M$. Let $g\in H^1(0,\tau )$ satisfying $g(0)\ne 0$. Then there exists a constant $C$, only depending on $(q_0,a_0)$, $\kappa$, $\Gamma$, $\tau$ and $g$, so that for any $(q,a)\in \mathscr{D}$ we have
\[
\|f\|_{L^2(M)}\le C\|\partial _\nu u\|_{H^1((0,\tau ),L^2(\Gamma ))}.
\]
Here $u=u(q,a,f,g)$ denotes the solution of the IBVP \eqref{w1.0}.
\end{theorem}

Set for simplicity $v=v(q,f,g)=u(q,0,f,g)$. That is $v$ is the solution of the IBVP
\begin{equation}\label{w1.0+}
\left\{
\begin{array}{lll}
 \partial _t^2 u - \Delta u + q(x)u = g(t)f(x) \quad &\mbox{in}\;   M \times (0,\tau), 
 \\
u = 0 &\mbox{on}\;  \partial M \times (0,\tau), 
\\
u(\cdot ,0) = 0,\quad \partial_t u (\cdot ,0) = 0.
\end{array}
\right.
\end{equation}

Using Duhamel's formula it is not hard to check that
\[
v(x,t)=\int_0^t g(t-s)w(x,s)ds,
\]
where $w=w(f)$ is the solution of the IBVP
\begin{equation}\label{w1.0bis}
\left\{
\begin{array}{lll}
 \partial _t^2 w - \Delta w + q(x)w =0 \quad &\mbox{in}\;   M \times (0,\tau), 
 \\
w = 0 &\mbox{on}\;  \partial M \times (0,\tau), 
\\
w(\cdot ,0) = f,\; \partial_t w (\cdot ,0) = 0.
\end{array}
\right.
\end{equation}

Let
\[
H^1_\ell ((0,\tau), L^2(\Gamma )) = \left\{u \in H^1((0,\tau), L^2(\Gamma )); \; u(0) = 0 \right\}
\]

and define the operator $S:L^2(\Gamma \times (0,\tau ) )\longrightarrow H^1_\ell ((0,\tau ) ,L^2(\Gamma ))$ by
\begin{equation*}
(Sh)(t)=\int_0^t g(t-s)h(s)ds.
\end{equation*}

We have seen in the proof of Theorem \ref{theorem-isp1} that $S$ is an isomorphism and
\[
\|h\|_{L^2(\Gamma \times (0,\tau) )}\leq \frac{\sqrt{2}}{\kappa |g(0)|}e^{\tau\frac{\|g'\|_{L^2((0,\tau))}^2}{|g(0)|^2}}\|Sh \|_{H^1((0,\tau ),L^2(\Gamma ))}.
\]
Whence
\begin{equation}\label{w2.0}
\|\partial_\nu w\|_{L^2(\Gamma \times (0,\tau) )}\le \frac{\sqrt{2}}{\kappa |g (0)|}e^{\tau\frac{\|g '\|_{L^2((0,\tau))}^2}{|g(0)|^2}}\| \partial_\nu v\|_{H^1 ((0,\tau ),L^2(\Gamma ))}.
\end{equation}

Let $\aleph>0$, assume that $\mathbf{d}(\Gamma )<\infty$ and let  $\tau >2\mathbf{d}(\Gamma )$. From Theorem \ref{theorem-O2} there exist three constants $C$, $\kappa$ and  $\epsilon _0$ so that for any $q\in L^\infty (M)$ with $\|q\|_{L^\infty (M)}\le \aleph$ we have
\begin{equation}\label{w3.0}
C\|f\|_{L^2(M)}\le e^{\kappa \epsilon}\| \partial _\nu w\|_{L^2(\Gamma \times (0,\tau ))}+\frac{1}{\epsilon}\|f\|_{H_0^1 (M)},\quad  \epsilon \ge \epsilon_0.
\end{equation}

Now \eqref{w2.0} in \eqref{w3.0} yields
\begin{equation}\label{w4.0}
C\|f\|_{L^2(M)}\le e^{\kappa \epsilon}\frac{\sqrt{2}}{\kappa |g (0)|}e^{\tau\frac{\|g '\|_{L^2((0,\tau))}^2}{|g(0)|^2}}\| \partial_\nu v\|_{H^1 ((0,\tau ),L^2(\Gamma ))}+\frac{1}{\epsilon}\|f\|_{H_0^1 (M)}
\end{equation}
for any $\epsilon \ge \epsilon_0$.

Let 
\[
\Psi (\rho )=|\ln \rho\, |^{-1}+\rho,\quad  \rho>0,
\]
and $\Psi (0)=0$. Then a standard minimization argument with respect to $\epsilon$ in \eqref{w4.0} enables us establishing the following result.

\begin{theorem}\label{theorem-w2.0}
Let $\aleph>0$, $R >0$ and  $\tau >2\mathbf{d}(\Gamma )$. Let $g\in H^1(0,\tau )$ satisfying $g(0)\ne 0$. Then there exists a constant $C>0$, only depending on $\aleph$, $R$, $\Gamma$, $\tau$ and $g$, so that for any $q\in L^\infty (M)$ with $\|q\|_{L^\infty (M)}\le \aleph$ and any $f\in H_0^1(M)$ satisfying $\|f\|_{H_0^1(M)}\le R$ we have
\[
\|f\|_{L^2(M)}\le C\Psi \left(  \| \partial_\nu v\|_{H^1 ((0,\tau ),L^2(\Gamma ))}\right),
\]
where $v=v(q,f,g)$ is the solution of the IBVP \eqref{w1.0+}.
\end{theorem}

\subsection{Determining the potential and the damping coefficient in a wave equation}

Introduce the  IBVP  for the wave equation
\begin{equation}\label{w1}
\left\{
\begin{array}{lll}
 \partial _t^2 u - \Delta u + q(x)u + a(x) \partial_t u = 0 \quad &\mbox{in}\;   M \times (0,\tau), 
 \\
u = 0 &\mbox{on}\;  \partial M \times (0,\tau), 
\\
u(\cdot ,0) = u_0,\quad \partial_t u (\cdot ,0) = u_1.
\end{array}
\right.
\end{equation}

Let $\aleph>0$ and recall that $\mathcal{H}_0=H_0^1(M) \oplus L^2(M)$. We have seen in Section 1 that, for any $(q,a)\in L^\infty (M)\oplus L^\infty (M)$, $\tau >0$ and $(u_0,u_1)\in \mathcal{H}_0$, the IBVP \eqref{w1} has a unique solution 
 \[ 
 u=u(q,a, (u_0,u_1))\in C([0,\tau ],H_0^1(M ))\] so that $\partial _tu\in C([0,\tau ],L^2(M ))$ and $\partial_\nu u\in L^2(\partial M \times (0,\tau))$. Moreover under the assumption
 \[
 \|(q,a)\|_{L^\infty (M)\oplus L^\infty (M)}  \le \aleph
 \]
we have
\begin{equation}\label{w2}
\|u\|_{C([0,\tau ],H_0^1(M))}+\|\partial _t u\|_{C([0,\tau ],L^2(M))}\le C\|(u_0,u_1)\|_{\mathcal{H}_0}
\end{equation}
and
\begin{equation}\label{w3}
\|\partial_\nu u\|_{L^2( \partial M \times (0,\tau))}\le C\|(u_0,u_1)\|_{\mathcal{H}_0}.
\end{equation}
Here $C=C(\aleph)$ is a nondecreasing function.

Define the initial-to-boundary operator $\Lambda (q,a)$ as follows
\[
\Lambda (q,a): (u_0,u_1)\in \mathcal{H}_0 \mapsto \partial_\nu u(q,a,(u_0,u_1))\in L^2(\Gamma \times (0,\tau )).
\] 

Let 
\[\mathcal{H}_1=(H_0^1(M)\cap H^2(M))\oplus H_0^1(M).\]
Observing that 
\[
\partial _t u(q,a,(u_0,u_1))=u(q,a, (u_1,\Delta u_0 - q u_0 -au_1))
\]
we easily obtain that $\Lambda (q,a)\in \mathscr{B}(\mathcal{H}_1,H^1((0,\tau ),L^2(\Gamma )))$. Furthermore  we get as a consequence of \eqref{w3}
\[
\|\Lambda (q,a)\|_{\mathscr{B}(\mathcal{H}_1,H^1((0,\tau ),L^2(\Gamma )))}\le C,
\]
where the constant $C$ is similar to that in \eqref{w3}.

Denote by $\mathscr{D}_0$ the set $\mathscr{D}$ in the case where $(q_0,a_0)=(q_0,0)$ with $q_0\ge 0$. Define then $\mathscr{D}_1 (\aleph)$ as the subset of $\mathscr{D}_0$ consisting in couples $(q,a)\in H^2(M)\oplus H^2(M)$ satisfying 
\[
\|(q,a)\|_{H^2(M)\oplus H^2(M)}\le \aleph.
\]
It is then clear that $\mathscr{D}_1(\aleph)$ is nonempty provided that $\aleph \ge \aleph_0$, for some $\aleph_0=\aleph(\beta)$.

\begin{theorem}\label{theorem-w1}
Assume that $(\Gamma ,\tau)$ geometrically control $M$ and let $\aleph \ge \aleph_0$. There exists a constant $C>0$, depending on $\aleph$ and $q_0$, so that for any $(q,a)\in \mathscr{D}_1(\aleph )$ we have 
\[
\|q-q_0\|_{L^2(M)}+\|a-0\|_{L^2(M)}\le C\|\Lambda (q,a)-\Lambda (q_0,0)\|_{\mathscr{B}(\mathcal{H}_1,H^1((0,\tau ),L^2(\Gamma )))}^{1/2}.
\]
\end{theorem}

\begin{proof}
Let $0\le \phi_1$  be the first eigenfunction of the operator $-\Delta +q_0$ with domain $H^2(M)\cap H_0^1(M)$. This eigenfunction is normalized by $\|\phi_1\|_{L^2(M)}=1$.  If 
\[
u_0=u(q_0,0, (\phi_1,i\sqrt{\lambda _1}\phi_1))=e^{i\sqrt{\lambda _1}\, t}\phi_1\quad \mbox{and}\quad u=u(q,a,(\phi_1,i\sqrt{\lambda _1}\, \phi_1))
\]
then $v= u-u_0$ is the solution of the following IBVP
\begin{equation}\label{w4}
\left\{
\begin{array}{lll}
 \partial _t^2 v - \Delta v + qv + a \partial_t v = -[(q-q_0)+i\sqrt{\lambda _1}a]e^{i\sqrt{\lambda _1}t}\phi_1 \quad &\mbox{in}\;   M\times (0,\tau), 
 \\
v = 0 &\mbox{on}\;  \partial M \times (0,\tau), 
\\
v(\cdot ,0) = 0,\quad \partial_t v (\cdot ,0) = 0.
\end{array}
\right.
\end{equation}
Bearing in mind that $(\Gamma ,\tau)$ geometrically control $M$ we get from Theorems \ref{theorem-w1.0} 
\[
\| \phi _1(q-q_0)\|_{L^2(M)}+\|\phi_1a\|_{L^2(M)}\le C\|\partial _\nu v\|_{H^1((0,\tau ),L^2(\partial M ))}.
\]
This inequality combined with Corollary \ref{corollary-wii1} yields
\[
\|q-q_0\|_{L^2(M)}+\|a-0\|_{L^2(M)}\le
C \|\partial _\nu v\|_{H^1((0,\tau ),L^2(\Gamma ))}^{1/2}
\]
which gives in a straightforward manner the expected result.
\end{proof}

Denote the sequence of eigenvalues, counted according to their multiplicity, of $A=-\Delta$ with domain $H^2(M)\cap H_0^1(M)$ by $0<\lambda _1<\lambda _2\leq \ldots \le \lambda _k \rightarrow \infty$.
 
Consider on $\mathcal{H}_0$ the operators 
\[ 
\mathcal{A}=\left( 
\begin{array}{cc}
0 & I  \\
-A  & 0  \\
 \end{array} 
 \right),\;\; D(\mathcal{A})=\mathcal{H}_1
 \]
 and $\mathcal{A}(q,a)=\mathcal{A}+\mathcal{B}(q,a)$ with $D(\mathcal{A}(q,a))=D(\mathcal{A})$, where
 \[ 
\mathcal{B}(q,a)=\left( 
\begin{array}{cc}
0 & 0 \\
-q & -a  \\
 \end{array} 
 \right) \in \mathscr{B}(\mathcal{H}_0).
 \]
 
From \cite[Proposition 3.7.6, page 100]{TW} $\mathcal{A}$ is skew-adjoint operator with $0\in \rho (\mathcal{A})$ and
\[ 
\mathcal{A}^{-1}=\left( 
\begin{array}{cc}
0 & -A^{-1}  \\
I  & 0  \\
 \end{array} 
 \right).
 \]
We note that, since $\mathcal{A}^{-1}:\mathcal{H} \rightarrow \mathcal{H}_1$ is bounded and the embedding $\mathcal{H}_1\hookrightarrow \mathcal{H}$ is compact, $\mathcal{A}^{-1}:\mathcal{H} \rightarrow \mathcal{H}$ is compact.

Also, from \cite[Proposition 3.7.6, page 100]{TW} $\mathcal{A}$ is diagonalizable and its spectrum consists in the sequence $(i\sqrt{\lambda _k})$.

Introduce the bounded operator 
\[
\mathcal{C}(q,a)=(i\mathcal{A}^{-1})(-i\mathcal{B}(q,a))(i\mathcal{A}^{-1}). 
\]
Let $s_k(\mathcal{C}(q,a))$, $k\ge 1$, denote the singular values of $\mathcal{C}(q,a)$, that is the eigenvalues of $[\mathcal{C}(q,a)^\ast \mathcal{C}(q,a)]^{\frac{1}{2}}$. In light of \cite[formulas (2.2) and (2.3), page 27]{GK} we have
\[
s_k(\mathcal{C}(q,a))\le \|\mathcal{B}(q,a)\|s_k(i\mathcal{A}^{-1})^2=\|\mathcal{B}(q,a)\|\lambda _k^{-1} ,
\]
where $\|\mathcal{B}(q,a)\|$ denote the norm of $\mathcal{B}(q,a)$ in $\mathscr{B}(\mathcal{H})$.

On the other hand referring to Weyl's asymptotic formula we get $\lambda _k =O(k^{2/n})$. Hence $\mathcal{C}_{q,a}$ belongs to the Shatten class $\mathcal{S}_p$ for any $p>n/2$, that is
\[
\sum_{k\ge 1}\left[s_k(\mathcal{C}(q,a))\right]^p<\infty .
\]
We get by applying \cite[Theorem 10.1, page 276]{GK} that the spectrum of $\mathcal{A}(q,a)$ consists in a sequence of eigenvalues $(\mu_k(q,a))$, counted according to their multiplicity, and the corresponding eigenfunctions $(\phi_k(q,a))$ form a Riesz basis of $\mathcal{H}$.

Fix $(q,a)$ and $k$. Set $\mu  =\mu _k(q,a)$ and $\phi= \phi_k(q,a) =(\varphi  ,\psi )\in \mathcal{H}_1$ be an eigenfunction associated to $\mu$. Then it is straightforward to check that
$\psi =\mu \varphi$ and $(-\Delta +q+a\mu +\mu ^2)\varphi=0$ in $M$. Since $-\Delta \varphi =f$ in $M$ with $f= - (q+a\mu +\mu ^2)\varphi$ we can use iteratively \cite[Corollary 7.11, page 158]{GT} (Sobelev embedding theorem) together with \cite[Theorem 9.15, page 241]{GT} in order to obtain that $\varphi \in W^{2,p}(M)$ for any $1<p<\infty$. In particular $\varphi ,|\varphi |^2\in W^{2,n}(M)\cap C^0(M)$. In other words $\varphi$ satisfies the assumption of Proposition \ref{proposition-wii2}.

Set, for $(q,a), (\tilde{q},\tilde{a})\in \mathscr{D}$,
\[
u=u(q,a, \phi)\quad \mbox{and}\quad \tilde{u}=u(\tilde{q},\tilde{a},\phi ).
\]
Then similarly to Theorem \ref{theorem-w1} we prove, where $v= \tilde{u}-u$,
\[
\| \varphi (\tilde{q}-q)\|_{L^2(M)}+\|\varphi (\tilde{a}-a)\|_{L^2(M)}\le C \|\partial _\nu v\|_{H^1((0,\tau ),L^2(\partial M ))}.
\]
This and Lemma \ref{lemma-wii1} yield 

\begin{theorem}\label{theorem-w2}
Assume that $(\Gamma ,\tau)$ geometrically control $M$ and fix $(q,a)\in \mathscr{D}$. Then there exists two constants $C>0$ and $\alpha >0$, depending of $(q,a)$, so that for any $( \tilde{q}, \tilde{a})\in \mathscr{D}$ we have
\[
\| \tilde{q}-q\|_{L^2(M)}+\|\tilde{a}-a\|_{L^2(M)}\le C\|\Lambda (\tilde{q},\tilde{a})-\Lambda (q,a)\|_{\mathscr{B}(\mathcal{H}_1,H^1((0,\tau ),L^2(\Gamma )))}^\alpha.
\]
\end{theorem}

\subsection{Determining the potential in a wave equation without geometric control condition}

Consider the IBVP
\begin{equation}\label{w5}
\left\{
\begin{array}{lll}
 \partial _t^2 u - \Delta u + q(x)u = 0 \quad &\mbox{in}\;   M \times (0,\tau), 
 \\
u = 0 &\mbox{on}\;  \partial M \times (0,\tau), 
\\
u(\cdot ,0) = u_0,\quad \partial_t u (\cdot ,0) = 0.
\end{array}
\right.
\end{equation}

From the preceding subsection the initial-to-boundary mapping
\[
\Lambda (q): u_0\in H_0^1(M)\cap H^2(M)\mapsto \partial_\nu u\in H^1((0,\tau ),L^2(\Gamma )),
\]
where $u=u(q,u_0)$ is the solution on the IBVP, defines a bounded operator. Moreover for any $\aleph>0$ there exists a constant $C>0$, depending of $\aleph$, so that for any $q\in L^\infty (M)$ satisfying $\|q\|_{L^\infty (M)}\le \aleph$ we have
\[
\|\Lambda (q)\|_{\mathscr{B}(H_0^1(M)\cap H^2(M),H^1((0,\tau ),L^2(\Gamma )))}\le C.
\]

\begin{theorem}\label{theorem-w3}
Let $\aleph >0$ and suppose that $\tau >2\mathbf{d}(\Gamma )$. There exists a constant $C>0$ so that for any $0\le q\in L^\infty (M)$, $\tilde{q}\in L^\infty (M)$ satisfying $q-\tilde{q}\in W^{1,\infty}(M)$ and
\[
\|q\|_{L^\infty(M)}\le \aleph,\quad  \|\tilde{q}\|_{L^\infty(M)}\le \aleph,\quad \|q-\tilde{q}\|_{W^{1,\infty}(M)}\le \aleph
\]
we have
\[
\|q-\tilde{q}\|_{L^2(M)}\le C\Phi \left(\|\Lambda (q)-\Lambda (\tilde{q})\|_{\mathscr{B}(H_0^1(M)\cap H^2(M),H^1((0,\tau ),L^2(\Gamma )))}\right),
\]
with $\Phi (\rho)=|\ln \rho |^{-1/(n+3)}+\rho$, $\rho >0$, and $\Phi (0)=0$.
\end{theorem}

\begin{proof}
Let $0\le q\in L^\infty (M)$ satisfying $\|q\|_{L^\infty (M)}\le \aleph$. Denote by $0<\lambda_1\le \lambda_2\ldots \le \lambda_k\ldots $ the sequence of eigenvalues of the operator $-\Delta +q$ with domain $H_0^1(M)\cap H^2(M)$. Let $(\phi_k)$ an orthonormal basis of $L^2(M)$ consisting in eigenfunctions, each $\phi_k$ is an eigenvalue for $\lambda_k$. Note that according to the usual elliptic regularity we have $\phi_k\in C^\infty (M)$ for each $k$.

By the Weyl's asymptotic formula and the min-max principle there exists a constant $\varkappa >1$, depending on $\aleph$ but not in $q$, so that
\begin{equation}\label{w6}
\varkappa^{-1}k^{2/n}\le \lambda_k \le \varkappa k^{2/n}.
\end{equation}

Set, for $\tilde{q}\in L^\infty (M)$ satisfying $\|\tilde{q}\|_{L^\infty (M)}\le \aleph$, 
\[
u=u(q,\phi_k)=\cos (\lambda_kt)\phi_k\quad \mbox{and}\quad \tilde{u}=u(\tilde{q},\phi_k).
\]
Then $v=\tilde{u}-u$ is the solution of the IBVP, where $g_k(t)=\cos (\sqrt{\lambda_k}t)$,
\begin{equation}\label{w7}
\left\{
\begin{array}{lll}
 \partial _t^2 u - \Delta u + \tilde{q}u = (\tilde{q}-q)\phi_k g_k(t) \quad &\mbox{in}\;   M \times (0,\tau), 
 \\
u = 0 &\mbox{on}\;  \partial M \times (0,\tau), 
\\
u(\cdot ,0) = 0,\quad \partial_t u (\cdot ,0) = 0.
\end{array}
\right.
\end{equation}

We have $\|g_k'\|^2_{L^2(0,\tau )}\le \lambda_k\tau$. Hence
\begin{equation}\label{w8}
\|g_k'\|^2_{L^2(0,\tau )}\le \varkappa \tau k^{2/n}
\end{equation}
by \eqref{w6}.

In the rest of this proof $C$ and $c$ denote generic constant only depending of $M$, $\aleph$, $\Gamma$ and $\tau$. From \eqref{w4.0} we have
\begin{equation}\label{w9}
C\|(\tilde{q}-q)\phi_k\|_{L^2(M)}\le e^{\kappa \epsilon}e^{ck^{2/n}}\| \partial_\nu v\|_{H^1 ((0,\tau ),L^2(\Gamma ))}+\frac{1}{\epsilon}\|(\tilde{q}-q)\phi_k\|_{H_0^1 (M)}
\end{equation}
for any $\epsilon \ge \epsilon_0$.

On the other hand 
\begin{align*}
\|(\tilde{q}-q)\phi_k\|_{H_0^1 (M)}&\le \|\widetilde{q}-q\|_{W^{1,\infty}(M)}\|\phi_k\|_{H_0^1 (M)}
\\
&\le 2N\sqrt{\lambda_k}
\\
&\le ck^{1/n}\;\; \mbox{by \eqref{w6}}.
\end{align*}
This in \eqref{w9} gives
\[
C\|(\tilde{q}-q)\phi_k\|_{L^2(M)}\le e^{\kappa \epsilon}e^{ck^{2/n}}\| \partial_\nu v\|_{H^1 ((0,\tau ),L^2(\Gamma ))}+\frac{k^{1/n}}{\epsilon},\quad  \epsilon \ge \epsilon_0.
\]
But we have by Cauchy-Schwarz's inequality 
\[
(\tilde{q}-q,\phi_k)^2\le \mbox{Vol}(M)\|(\tilde{q}-q)\phi_k\|_{L^2(M)}.
\]
Whence
\[
C(\tilde{q}-q,\phi_k)^2\le e^{\kappa \epsilon}e^{ck^{2/n}}\| \partial_\nu v\|_{H^1 ((0,\tau ),L^2(\Gamma ))}+\frac{k^{\frac{1}{n}}}{\epsilon},\quad \epsilon \ge \epsilon_0.
\]
Also
\begin{align*}
\|\tilde{q}-q\|_{L^2(M)}^2 &=\sum_{k\le \ell }(\tilde{q}-q,\phi _k)^2+\sum_{k>\ell}(\tilde{q}-q,\phi _k)^2
\\
&\le \sum_{k\le \ell}(\tilde{q}-q,\phi _k)^2 +\frac{1}{\lambda _{\ell+1}}\sum_{k>\ell}\lambda _k(\tilde{q}-q,\phi _k)^2
\\
&\le \sum_{k\le \ell}(\tilde{q}-q,\phi _k)^2 +\frac{N^2}{(\ell+1)^{2/n}}.
\end{align*}
Thus
\begin{equation}\label{w10}
C\|\tilde{q}-q\|_{L^2(M)}^2\le \ell e^{\kappa \epsilon}e^{c\ell ^{2/n}}\| \partial_\nu v\|_{H^1 ((0,\tau ),L^2(\Gamma ))}+\frac{1}{(\ell+1)^{2/n}}+\frac{\ell ^{1+1/n}}{\epsilon}.
\end{equation}
Let $s\geq 1$ be a real number and let $\ell $ be the unique integer so that $\ell \leq s<\ell+1$. Then \eqref{w10} with that $\ell$ yields
\begin{equation}\label{w11}
C\|\tilde{q}-q\|_{L^2(M)}^2\le s e^{\kappa \epsilon}e^{cs^{2/n}}\| \partial_\nu v\|_{H^1 ((0,\tau ),L^2(\Gamma ))}+\frac{1}{s^{2/n}}+\frac{s^{1+1/n}}{\epsilon}.
\end{equation}
We then get by taking $\epsilon =s^{3/n+1}$ in \eqref{w11}, where $s_0=\max \left( 1, \epsilon_0^{n/(n+3)}\right)$,
\[
C\|\tilde{q}-q\|_{L^2(M)}^2 \le \frac{1}{s^{2/n}}+e^{c s^{2/n+1}}e^{\kappa s^{3/N+1}}\| \partial_\nu v\|_{H^1 ((0,\tau ),L^2(\Gamma ))},\quad s\ge s_0.
\]
Therefore
\[
C\|\tilde{q}-q\|_{L^2(M)}^2 \le \frac{1}{s^{2/n}}+e^{cs^{3/n+1}}\| \partial_\nu v\|_{H^1 ((0,\tau ),L^2(\Gamma ))},\quad s\ge s_0
\]
or equivalently
\[
C\|\tilde{q}-q\|_{L^2(M)}\le \frac{1}{s^{1/n}}+e^{cs^{3/n+1}}\| \partial_\nu v\|_{H^1 ((0,\tau ),L^2(\Gamma ))},\quad s\ge s_0.
\]
We end up getting the expected inequality by minimizing with respect to $s$.
\end{proof}

\subsection{Inverse source problem for the heat equation}

Consider the following IBVP for the heat equation
\begin{equation}\label{w12}
\left\{
\begin{array}{lll}
 \partial _t u - \Delta u + q(x)u  = g(t)f(x) \quad &\mbox{in}\;   M\times (0,\tau ),
 \\
u = 0 &\mbox{on}\;  \partial M\times (0,\tau ),
\\
u(\cdot ,0) = 0
\end{array}
\right.
\end{equation}
and set $Q = M\times (0,\tau )$.

We recall that the anisotropic Sobolev space $H^{2,1}(Q)$ is given as follows
\[
H^{2,1}(Q)=L^2((0,\tau ),H^2(M))\cap H^1((0,\tau ), L^2(M)).
\]

From classical parabolic regularity theorems for any $f\in L^2(M)$, $g\in L^2(0,\tau )$ and $q\in L^\infty (M)$ the IBVP \eqref{w12} has a unique solution 
\[
u=u(q,f,g)\in H^{2,1}(Q).
\] 
Furthermore if $\aleph >0$ then there exists a constant $C>0$ so that
\begin{equation}\label{w13}
\|u\|_{H^{2,1}(Q)}\le C\|g\|_{L^2(0,\tau )}\|f\|_{L^2(M)}
\end{equation}
for any $q\in L^\infty (M)$ satisfying $\|q\|_{L^\infty (M)}\le \aleph$.

If in addition $g\in H^1(0,\tau )$ then it is not hard to check that $\partial _tu$ is the solution of the IBVP \eqref{w12} with $g$ substituted by $g'$. Hence $\partial _t u\in H^{2,1}(Q)$ and
\begin{equation}\label{w14}
\|\partial _tu\|_{H^{2,1}(Q)}\le C\|g'\|_{L^2(0,\tau )}\|f\|_{L^2(M )}
\end{equation}
for any $q\in L^\infty (M)$ satisfying $\|q\|_{L^\infty (M)}\le \aleph$, where $C$ is the constant  in \eqref{w13}.

We derive that $\partial _\nu u$ is well defined as an element of $H^1((0,\tau ),L^2(\Gamma))$. Therefore by \eqref{w13}, \eqref{w14} and the continuity of the trace operator on $\Gamma$ we have
\[
\|\partial _\nu u\|_{H^1((0,\tau ),L^2(\Gamma ))}\le C\|g\|_{H^1(0,\tau )}\|f\|_{L^2(M )},
\]
where the constant $C$ is as in \eqref{w13}.

The  following result will be useful in the sequel.

\begin{proposition}\label{proposition-w1}
Let $\aleph >0$. There exist two constants $c>0$ and  $C>0$ so that for any $q\in L^\infty (M)$ satisfying $\|q\|_{L^\infty (M)}\le \aleph$, $f\in H_0^1(M)$ and $g\in H^1(0,\tau )$ with $g(0)\ne 0$ we have
\begin{equation}\label{w15}
C\|f\|_{L^2(M)}\leq \frac{1}{\sqrt{\epsilon}}\|f\|_{H_0^1(M)}+\frac{1}{|g(0)|}e^{\tau\|g'\|_{L^2(0,\tau)}^2/|g(0)|^2}e^{c\epsilon}\|\partial _\nu u\|_{H^1((0,\tau ),L^2(\Gamma ))}
\end{equation}
for any $\epsilon \ge 1$, where $u=u(q,f,g)$ is the solution of the IBVP \eqref{w12}.
\end{proposition}

\begin{proof}
Pick $q\in L^\infty (M)$ satisfying $\|q\|_{L^\infty (M)}\le \aleph$, $f\in H_0^1(M)$ and $g\in H^1(0,\tau )$ with $g(0)\ne 0$. We may assume without loss of generality that $q\ge 0$. This is achieved by substituting $u$ by $ue^{-\aleph t}$, which is the solution of the IBVP \eqref{w12} when $q$ is replaced by $q+\aleph$.

Let $v=v(q,f)\in H^{2,1}(Q)$ be the unique solution of the IBVP
\begin{equation*}
\left\{
\begin{array}{lll}
 \partial _t v - \Delta v + q(x)v  = 0 \;\; &\mbox{in}\;   M\times (0,\tau), 
 \\
v = 0 &\mbox{on}\;  \partial M\times (0,\tau ),
\\
v(\cdot ,0) = f.
\end{array}
\right.
\end{equation*}
Then $\partial _\nu v$ is well defined as an element of $L^2(\Gamma \times (0,\tau ))$. As for the wave equation we have
\[
\partial _\nu u|_{\Gamma }(\cdot ,t)=\int_0^t g(t-s)\partial _\nu v|_{\Gamma}(\cdot ,s)ds.
\]
Therefore
\begin{equation}\label{w16}
\|\partial _\nu v\|_{L^2(\Gamma \times (0,\tau ) )}\le \frac{\sqrt{2}}{|g(0)|}e^{\tau\|g'\|_{L^2((0,\tau))}^2/|g(0)|^2}\|\partial _\nu u\|_{H^1((0,\tau ),L^2(\Gamma ))}.
\end{equation}

From the final time observability inequality in Theorem \ref{theorem-O3} we have
\begin{equation}\label{w17}
\|v(\cdot ,\tau )\|_{L^2(M )}\le K\|\partial _\nu v\|_{L^2(\Gamma \times (0,\tau ) )},
\end{equation}
for some constant $K>0$ independent of $q$ and $f$.

A combination of \eqref{w16} and \eqref{w17} yields
\begin{equation}\label{w18}
C\|v(\cdot ,\tau )\|_{L^2(M )}\le \frac{1}{|g(0)|}e^{\tau\|g'\|_{L^2((0,\tau))}^2/|g(0)|^2}\|\partial _\nu u\|_{H^1((0,\tau ),L^2(\Gamma ))}.
\end{equation}

Denote by $0 <\lambda _1\le \lambda _2 \le  \ldots \le \lambda _k \rightarrow \infty$ the sequence of eigenvalues of  the $-\Delta +q$ with domain $H_0^1(M)\cap H^2(M )$. Let $(\phi _k)$ be a sequence of eigenfunctions, each $\phi_k$ is associated to $\lambda _k$, so that $(\phi _k)$ form an orthonormal basis of $L^2(M)$. 

We have
\[
v(\cdot ,\tau )=\sum_{\ell \geq 1}e^{-\lambda _k\tau}(f,\phi_\ell )\phi_\ell ,
\]
where $(\cdot ,\cdot )$ is the usual scalar product on $L^2(M)$. Hence
\[
(f,\phi_\ell )^2\le e^{2\lambda _\ell \tau} \|v(\cdot ,\tau )\|_{L^2(M)}^2,\quad \ell \ge 1.
\]
Whence
\[
\sum_{\ell =1}^k(f,\phi_\ell )^2 \le ke^{2\lambda _k\tau}\|v(\cdot ,\tau )\|_{L^2(M)}^2
\]
for any integer $k\geq 1$.

This and the fact that $\left(\sum_{\ell\ge 1}\lambda _\ell (\cdot ,\phi_\ell )_{L^2(\Omega )}^2\right)^{1/2}$ is an equivalent norm on $H_0^1(M)$ lead
\begin{align*}
\|f\|_{L^2(M)}^2 &= \sum_{\ell =1}^k(f,\phi_\ell )^2+\sum_{\ell \ge k+1}(f,\phi_\ell )^2
\\
&\le \sum_{\ell =1}^k(f,\phi_\ell )^2+\frac{1}{\lambda_{k+1}}\sum_{\ell \ge k+1}\lambda_\ell (f,\phi_\ell )^2
\\
& \le ke^{2\lambda _k\tau}\|v(\cdot ,\tau )\|_{L^2(M)}^2+ \frac{1}{\lambda_{k+1}}\|f\|_{H_0^1(M )}^2.
\end{align*}

Until the end of this proof $C$ and $c$ denote generic constants independent of $q$, $f$ and $g$.
 
We get from inequality \eqref{w6} 
\begin{equation}\label{w19}
C\|f\|_{L^2(M)}^2\le ke^{ck^{2/n}}\|v(\cdot ,\tau )\|_{L^2(M)}^2+ \frac{1}{(k+1)^{2/n}}\|f\|_{H_0^1(M)}^2.
\end{equation}

Let $\epsilon \ge 1$ and $k\ge 1$ be the unique integer so that $k\le \epsilon^{n/2}<k+1$. We obtain in a straightforward manner from \eqref{w19}
\begin{equation}\label{w20}
C\|f\|_{L^2(M)}^2\leq e^{c\epsilon}\|v(\cdot ,\tau )\|_{L^2(M)}^2+ \frac{1}{\epsilon}\|f\|_{H^1(M)}^2.
\end{equation}
Then \eqref{w18} in \eqref{w20} gives the expected inequality.
\end{proof}

Minimizing the right hand side of \eqref{w15} with respect to $\epsilon$ we obtain the following result in which $\Phi (\rho )=\left| \ln \rho \right|^{-1/2}+\rho$ for $\rho >0$ and $\Phi (0)=0$

\begin{corollary}\label{corollary-w1}
Let $\aleph >0$, $q\in L^\infty (M)$ and $g\in H^1(0,\tau )$ satisfying $g(0)\ne 0$. There exists a constant $C>0$, depending of $\aleph$, $q$ and $g$, so that for any $f\in H_0^1(\Omega )$ with $\|f\|_{H_0^1(M)}\le \aleph$ we have
\[
C\|f\|_{L^2(M )}\le \Phi \left(\|\partial _\nu u \|_{H^1((0,\tau ) ,L^2(\Gamma))}\right),
\]
where $u=u(q,f,g)$ is the solution of the IBVP \eqref{w12}.
\end{corollary}

\subsection{Determining the zeroth order coefficient in a heat equation}

Consider the IBVP
\begin{equation}\label{w21}
\left\{
\begin{array}{lll}
 \partial _t u - \Delta u + q(x)u  = 0 \quad &\mbox{in}\;   M\times (0,\tau ), 
 \\
u = 0 &\mbox{on}\;  \partial M\times (0,\tau ), 
\\
u(\cdot ,0) = u_0.
\end{array}
\right.
\end{equation}
Again, with reference to classical regularity theorems we have that, for $q\in L^\infty (M)$ and $u_0\in H_0^1(M)$, the IBVP\eqref{w21}  has unique solution $u=u(q,u_0)\in H^{2,1}(Q)$. Furthermore for any $\aleph>0$ there exists a constant $C>0$ so that
\begin{equation}\label{w22}
\|u(q,u_0)\|_{H^{2,1}(M\times (0,\tau ))}\le C \|u_0\|_{H_0^1(M )}
\end{equation}
for any $q\in L^\infty (M)$ satisfying $\|q\|_{L^\infty (M)}\le \aleph$.

Define 
\[
\mathcal{H}_0(M)=\{w\in H_0^1(M);\; \Delta w\in H_0^1(M)\}
\] 
that we equip with its natural norm
\[
\|w\|_{\mathcal{H}_0(M)}=\| w\|_{H_0^1(M)}+\|\Delta w\|_{H_0^1(M)}.
\]

If $q\in W^{1,\infty}(M)$ and $u_0\in \mathcal{H}_0 (M)$ then it is straightforward to check that 
\[
\partial _tu(q,u_0)=u(q,\Delta u_0-qu_0).
\]
We get by applying \eqref{w22} with $u_0$ substituted by $\Delta u_0-qu_0$
\begin{equation}\label{w23}
\|\partial _tu\|_{H^{2,1}(M\times (0,\tau))}\le C \|u_0\|_{\mathcal{H}_0(M)}
\end{equation}
for any $q\in W^{1,\infty} (M)$ satisfying $\|q\|_{W^{1,\infty}(M)}\le \aleph$, where the constant $C$ is independent of $q$.

Bearing in mind that the trace operator 
\[
w\in H^{2,1}(Q)\mapsto \partial _\nu w\in L^2(\Gamma \times (0,\tau ) )
\]
is bounded we obtain that $\partial _\nu u \in H^1((0,\tau ),L^2(\Gamma ))$ provided that $u_0\in \mathcal{H}_0(M)$ and $q\in W^{1,\infty}(M )$. Further we get from \eqref{w22} and \eqref{w23}
\[
\|\partial _\nu u\|_{H^1((0,\tau ),L^2(\Gamma ))}\le C \|u_0\|_{\mathcal{H}_0(M)}
\]
for any $q\in W^{1,\infty} (M)$ satisfying $\|q\|_{W^{1,\infty} (M)}\le \aleph$, where the constant $C$ is independent of $q$.

That is we proved that the operator 
\[
\mathcal{N}(q):u_0\in \mathcal{H}_0(M)\mapsto \partial _\nu u\in H^1((0,\tau ),L^2(\Gamma ))
\] 
is bounded and
\[
\|\mathcal{N}(q)\|_{\mathscr{B}(\mathcal{H}_0(M),H^1((0,\tau ),L^2(\Gamma )))}\le C 
\]
for any $q\in W^{1,\infty} (M)$ satisfying $\|q\|_{W^{1,\infty} (M)}\le N$, where the constant $C$ is independent of $q$.

Henceforward for convenience $\|\mathcal{N}(\tilde{q})-\mathcal{N}(q)\|_{\mathscr{B}(\mathcal{H}_0(M),H^1((0,\tau ),L^2(\Gamma )))}$  is simply denoted by $\|\mathcal{N}(\tilde{q})-\mathcal{N}(q)\|$.

\begin{theorem}\label{theorem-w4}
Let $\aleph>0$. There exists a constant $C>0$ so that for any $q,\, \tilde{q} \in W^{1,\infty}(M)$ satisfying 
\[
\|q\|_{W^{1,\infty} (M)}\le \aleph,\quad \|\tilde{q}\|_{W^{1,\infty} (M)}\le \aleph,
\]
we have 
\[
C\|\widetilde{q}-q\|_{L^2(M )}\le \Theta \left(\|\mathcal{N}(\tilde{q})-\mathcal{N}(q)\|\right).
\]
Here $\Theta (\rho )=\left| \ln \rho \right|^{-1/(1+4n)}+\rho $ for $\rho >0$ and $\Theta (0)=0$.
\end{theorem}

\begin{proof} 
Let $q,\, \tilde{q} \in W^{1,\infty}(M)$ satisfying 
\[
\|q\|_{W^{1,\infty} (M)}\le \aleph,\quad \|\tilde{q}\|_{W^{1,\infty} (M)}\le \aleph.
\] 
As in the preceding subsection we may assume without loss of generality that $q\ge 0$. 

Denote by $0 <\lambda _1\le \lambda _2 \le  \ldots \le \lambda _k \rightarrow \infty$ the sequence of eigenvalues of  the operator $-\Delta +q$ with domain $H_0^1(M)\cap H^2(M)$. Let $(\phi _k)$ a sequence of the corresponding eigenfunctions so that $(\phi _k)$ form an orthonormal basis of $L^2(M)$. 

Taking into account that $u(q,\phi _k)=e^{-\lambda _kt}\phi _k$ we obtain that 
\[
v=u(\tilde{q},\phi_k)-u(q,\phi_k)
\] is the solution of the IBVP
\begin{equation*}
\left\{
\begin{array}{lll}
 \partial _t v - \Delta v + q(x)v  = (\tilde{q}-q)\phi _ke^{-\lambda_kt} \quad &\mbox{in}\;   M\times (0,\tau ),
 \\
uv= 0 &\mbox{on}\;  \partial M\times (0,\tau ),
\\
v(\cdot ,0) = 0.
\end{array}
\right.
\end{equation*}

Therefore
\[
\mathcal{N}(\tilde{q})(\phi_k)-\mathcal{N}(q)(\phi_k)=\partial _\nu v
\]
from which we deduce
\[
\| \partial _\nu v\|_{H^1((0,\tau ),L^2(\Gamma ))}\le C\lambda_k  \|\mathcal{N}(\widetilde{q})-\mathcal{N}(q)\|.
\]
Here and henceforth $C$ and $c$ denote generic constants independent of $q$ and $\tilde{q}$.

As in the preceding subsection we get from \eqref{w15}
\begin{equation}\label{w24}
C|(\tilde{q}-q,\phi_k)|\le\frac{\sqrt{\lambda_k}}{\sqrt{\epsilon}}+e^{\tau \lambda _k^2}e^{c\epsilon}\lambda _k^2 \|\mathcal{N}(\tilde{q})-\mathcal{N}(q)\|
\end{equation}
for any $\epsilon \ge 1$, where we used the inequality $\|(\tilde{q}-q)\phi_k\|_{H_0^1(M )}\leq C\sqrt{\lambda _k}$.

Inequality \eqref{w24} then gives
\begin{equation}\label{w25}
C\sum_{k=1}^\ell |(\tilde{q}-q,\phi_k)_{L^2(M )}^2\le \frac{\ell\lambda _\ell}{\epsilon}+ \ell e^{c\lambda _\ell^2}e^{c\epsilon}\|\mathcal{N}(\widetilde{q})-\mathcal{N}(q)\|^2
\end{equation}
for any arbitrary integer $\ell \ge 1$.

Similarly to the proof of Theorem \ref{theorem-w3} inequality \eqref{w25} yields
\[
C\|\tilde{q}-q\|_{L^2(M)}^2\le \frac{s^{1+2/n}}{\epsilon}+\frac{1}{s^{2/n}}+e^{c s^{1+4/n}}e^{c\epsilon}\|\mathcal{N}(\tilde{q})-\mathcal{N}(q)\|^2,\quad s\geq 1 .
\]
The proof is then completed in the same manner like that of Theorem \ref{theorem-w3}. 
\end{proof}

\section{Determining a boundary coefficient in a wave equation}

\subsection{Inverse source problem for the wave equation with boundary damping}

In this subsection $\Omega =(0,1)\times (0,1)$ and
\begin{align*}
&\Gamma _0=((0,1)\times \{1\})\cup (\{1\}\times (0,1)),
\\
&\Gamma _1=((0,1)\times \{0\})\cup (\{0\}\times (0,1)).
\end{align*}

 Consider the IBVP
\begin{equation}\label{e3}
\left\{
\begin{array}{lll}
 \partial _t^2 u - \Delta u = \lambda (t)w \quad &\mbox{in}\;   \Omega \times (0,\tau), 
 \\
u = 0 &\mbox{on}\;  \Gamma _0 \times (0,\tau), 
\\
\partial _\nu u +a\partial _tu=0 &\mbox{on}\;  \Gamma _1 \times (0,\tau), 
\\
u(\cdot ,0) = 0,\quad \partial_t u (\cdot ,0) = 0.
\end{array}
\right.
\end{equation}

Fix $\frac{1}{2}<\alpha \le 1$ and let
\[
\mathscr{A}=\{ b=(b_1,b_2)\in C^\alpha ([0,1])\oplus C^\alpha ([0,1]),\; b_1(0)=b_2(0),\; b_j\geq 0\}. 
\]

Let $V=\{ u\in H^1(\Omega );\; u=0\; \textrm{on}\; \Gamma _0\}$ and define on $V\oplus L^2(\Omega )$ the operator $A_a$, $a\in \mathscr{A}$,  by
\begin{align*}
&A_a= (w,\Delta v),
\\
&D(A_a)=\{ (v,w)\in V\oplus V;\; \Delta v\in L^2(\Omega )\; \textrm{and}\; \partial _\nu v=-aw\; \textrm{on}\; \Gamma _1\}.
\end{align*}

We are going to apply Theorem \ref{theorem-isp2} with $H=V\oplus L^2(\Omega )$, $H_1=D(A_a)$ equipped with its graph norm and $Y=L^2(\Gamma _1)$.

Denote by $H_{-1}$ the dual of $H_1$ with respect to the pivot space $H$.

If $(0,w)\in H_{-1}$ and $\lambda \in H^1(0,\tau )$ then  the IBVP \eqref{e3} has a unique solution $u(w)$ so that $(u(w),\partial _tu(w)) \in C([0,\tau ]; V\oplus L^2(\Omega ))$ and $\partial_\nu  u(w)|{_{\Gamma _1\times (0,\tau)}}\in L^2(\Gamma _1\times (0,\tau ) )$.

Taking into account that $\{0\}\times V' \subset H_{-1}$, where $V'$ is the dual space of $V$, we obtain the following consequence of Theorem \ref{theorem-isp2}.

\begin{proposition}\label{proposition-e1}
There exists a constant $C>0$ so that for any $\lambda \in H^1(0,\tau )$ and $w\in V'$ we have
\begin{equation}\label{e4}
 \|w\|_{V'}\leq C|\lambda (0)|e^{\tau \|\lambda '\|^2_{L^2(0,\tau )}/|\lambda (0)|^2}\|\partial _\nu u_w\|_{L^2 (\Gamma _1\times (0,\tau ))}.
\end{equation}
\end{proposition}

\subsection{Determining the boundary damping coefficient in a wave equation}

Let $\Omega$ and $\Gamma_i$, $i=1,2$ as in the preceding subsection and consider the IBVP
\begin{equation}\label{e5da}
\left\{
\begin{array}{lll}
 \partial _t^2 u - \Delta u = 0 \quad &\mbox{in}\;   \Omega \times (0,\tau), 
 \\
u = 0 &\mbox{on}\;  \Gamma _0 \times (0,\tau), 
\\
\partial _\nu u +a\partial _tu=0 &\mbox{on}\;  \Gamma _1 \times (0,\tau), 
\\
u(\cdot ,0) = u_0,\quad \partial_t u (\cdot ,0) = u_1.
\end{array}
\right.
\end{equation}

For $(u_0,u_1)\in H_1$ the IBVP \eqref{e5da} admits a unique solution $u=u(a,(u_0,u_1))$ so that
\[
(u_a,\partial _tu_a)\in C([0,\infty ),H_1)\cap C^1([0,\infty ),H).
\]

Fix $0<\underline{a} \le \aleph$ and set
\[
\underline{\mathscr{A}} =\{ b=(b_1,b_2)\in \mathscr{A}\cap H^1(0,1)\oplus  H^1(0,1);\; \underline{a} \le b_1,\, b_2,\; \|b\|_{H^1(0,1)\oplus  H^1(0,1)}^2\leq \aleph\}.
\]

Let $\mathcal{U}_0$ given by 
\[
\mathcal{U}_0=\{ v\in V;\; \Delta v\in L^2(\Omega )\; \textrm{and}\; \partial _\nu v=0\; \textrm{on}\; \Gamma _1\}
\]
and observe that $\mathcal{U}_0\times \{0\}\subset H_1$ for any $a\in \mathscr{A}$. We endow $\mathcal{U}_0$ with the norm
\[
\|u_0\|_{\mathcal{U}_0}=\left(\|u_0\|_V^2+\|\Delta u_0\|_{L^2(\Omega )}^2\right)^{1/2}.
\]

Define the initial-to-boundary operator
\[
\Lambda (a) :u_0\in \mathcal{U}_0 \mapsto \partial_\nu u\in L^2(\Gamma _1\times (0,\tau )),
\]
where $u=u(a,(u_0,u_1)$ is the solution of the IBVP \eqref{e5da}. Then 
\[
\Lambda (a)\in \mathscr{B}(\mathcal{U}_0,L^2(\Gamma _1\times (0,\tau))).
\]

Henceforward for convenience the norm of $\Lambda(a)-\Lambda (0)$ in $\mathscr{B}(\mathcal{U}_0,L^2(\Gamma _1\times (0,\tau)))$ will simply denoted by $\| \Lambda(a)-\Lambda (0)\|$.

The following H\"older stability inequality improve the result in \cite{AC3}.

\begin{theorem}\label{theorem-e2}
Let $\delta \in (0,1)$. There exists  a constant $C$ only depending of $\delta$, $\underline{a}$ and $\aleph$ so that
\begin{equation}\label{e6}
 \|a-0\|_{L^2(0,1)\oplus L^2(0,1)} \le C \| \Lambda(a)-\Lambda (0)\|^{\delta/[2(2+\delta)]}
\end{equation}
for each $a\in \underline{\mathscr{A}}$.
\end{theorem}

\begin{proof}
We first observe that $u(a)$ is also the unique solution of 
\begin{equation*}
\left\{
\begin{array}{lll}
\displaystyle  \int_\Omega u''(t)vdx=\int_\Omega \nabla u(t)\cdot \nabla vdx-\int_{\Gamma _1}au'(t)v,\quad  v\in V.
 \\
\\
 u(0)=u_0,\quad u'(0)=u_1.
\end{array}
\right.
\end{equation*}

Therefore $u=u(a)-\underline{u}$, where $\underline{u}=u(0,(u_0,u_1))$, is the solution of the following problem
\begin{equation}\label{e7}
\left\{
\begin{array}{lll}
\displaystyle \int_\Omega u''(t)vdx=\int_\Omega \nabla u(t)\cdot \nabla vdx-\int_{\Gamma _1}au'(t)v -\int_{\Gamma _1}au'(0)(t)v,\quad  v\in V.
 \\ \\
 u(0)=0,\quad u'(0)=0.
\end{array}
\right.
\end{equation}

For $k$, $\ell \in \mathbb{Z}$ set 
\begin{align*}
&\lambda_{k\ell}=\left[ \left(k+\frac{1}{2}\right)^2+\left(\ell +\frac{1}{2}\right)^2\right]\pi ^2
\\
&\phi_{k\ell}(x,y)=2\cos \left(\left(k+\frac{1}{2}\right)\pi x\right)\cos \left(\left(\ell +\frac{1}{2}\right)\pi y\right).
\end{align*}
and observe that $\underline{u}=\cos(\sqrt{\lambda_{k\ell}}\, t)\phi_{k\ell}$ when $(u_0,u_1)=(\phi_{k\ell},0)$. 

Fix $k$ and $\ell$ and set $\lambda (t)=\cos(\sqrt{\lambda_{k\ell}}\, t)$. Define $w(a)\in V'$ by
\[
w(a)(v)=-\sqrt{\lambda_{k\ell}}\int_{\Gamma _1}a\phi_{k\ell}v.
\]

Whence, \eqref{e7} becomes
\begin{equation*}
\left\{
\begin{array}{lll}
\displaystyle \int_\Omega u''(t)vdx=\int_\Omega \nabla u(t)\cdot \nabla vdx-\int_{\Gamma _1}au'(t)v +\lambda (t)w(a)(v),\quad v\in V.
 \\ \\
 u(0)=0,\;\; u'(0)=0.
\end{array}
\right.
\end{equation*}
In other words $u$ is the solution of \eqref{e3} with $w=w(a)$. We find by applying Proposition \ref{proposition-e1}
\begin{equation}\label{e8}
\|w(a)\|_{V'}\leq Ce^{\lambda_{k\ell}\tau ^2}\| \partial _\nu u\|_{L^2(\Gamma _1\times (0,\tau ))}.
\end{equation}
By noting that $(a_1\otimes a_2)\phi_{k\ell}\in V$ even if  $a_1\otimes a_2\not\in V$ we obtain
\begin{align}
&a_1(0)\left| \int_{\Gamma _1}(a\phi_{k\ell})^2d\sigma \right| =\frac{1}{\sqrt{\lambda_{k\ell}}}\left| w(a) ((a_1\otimes a_2)\phi_{k\ell})\right| \label{e9}
\\
&\hskip 5 cm\le \frac{1}{\sqrt{\lambda_{k\ell}}}\|w(a)\|_{V'}\|(a_1\otimes a_2)\phi_{k\ell}\|_V,\nonumber
\end{align}
where we used that $a_1(0)=a_2(0)$ and
\begin{equation}\label{e10}
\|(a_1\otimes a_2)\phi_{k\ell}\|_V \le C_0\sqrt{\lambda _{kl}}\| a_1\otimes a_2\|_{H^1(\Omega )}.
 \end{equation}
 Here and henceforth $C_0$ is a generic constant independent of $a$ and $\phi_{k\ell}$.

Now a combination of \eqref{e8}, \eqref{e9} and \eqref{e10} yields
\begin{align*}
a_1(0)\Big(\|a_1\phi _k\|_{L^2(0,1)}^2+&\|a_2\phi _\ell\|_{L^2(0,1)}^2\Big)
\\
&\le C\| a_1\|_{H^1(0,1)}\| a_2\|_{H^1(0,1)}e^{\lambda_{k\ell}\frac{\tau ^2}{2}}\| \partial _\nu u\|_{L^2(\Gamma _1\times (0,\tau ))},
\end{align*}
where $\phi_k (s)=\sqrt{2}\cos \left(\left(k+1/2\right)\pi s\right)$. This, $\underline{a}\leq a_j(0)$ and $\|a_j\|_{H^1(0,1)}\leq \aleph$ imply
\begin{equation*}
\|a_1\phi _k\|_{L^2(0,1)}^2+\|a_2\phi _\ell\|_{L^2(0,1)}^2\le C_0e^{\lambda_{k\ell}\frac{\tau ^2}{2}}\| \partial _\nu u\|_{L^2(\Gamma _1\times (0,\tau ))}.
\end{equation*}
Hence, where $j=1$ or $2$,
\begin{equation}\label{ajj}
\|a_j\phi _k\|_{L^2(0,1)}^2\le C_0e^{k^2\tau ^2\pi^2}\| \partial _\nu u\|_{L^2(\Gamma _1\times (0,\tau ))}.
\end{equation}
Let $\delta \in (0,1)$ be fixed. Observing that $\phi_0(s)\sim \pi (s-1)/2$ as $s\rightarrow 1$ we deduce that $|\phi_0|^{-\delta} \in L^1(0,1)$. Then we obtain by following the proof of  Lemma \ref{lemma-wii2}
\begin{align}
\|a_j\|_{L^2(0,1)} &\leq \||\phi_0|^{-\delta}\|_{L^1(0,1)}^{1/(2+\delta)}
\|a_j\|_{L^\infty(0,1)}^{2/(2+\delta)}\|a_j\phi _0\|_{L^2(0,1)}^{\delta/(2+\delta)} \label{aaa} 
\\
&\leq C \|a_j\phi _0\|_{L^2(0,1)}^{\delta/(2+\delta)}. \nonumber
\end{align}
A combination of inequalities \eqref{aaa} and \eqref{ajj} with $k=0$ yields
\begin{equation*}
\|a_j\|_{L^2(0,1)}\le
C  \|\partial _\nu u\|_{L^2(\Gamma _1\times (0,\tau ))}^{\delta/[2(2+\delta)]}.
\end{equation*}
This achieves the proof of the expected inequality.
\end{proof}

\end{document}